\documentclass{siamart190516}

\usepackage{preambule}

\usepackage{chngcntr}
\counterwithin*{equation}{section}
\renewcommand{\theequation}{\arabic{section}.\arabic{equation}}

\title{A general error analysis for randomized low-rank approximation methods \thanks{\funding{This work was funded by ISAE-SUPAERO.}}}

\author{Youssef Diouane\thanks{Department of Mathematics and Industrial Engineering, Polytechnique Montr\'eal, QC, Canada
  (\email{youssef.diouane@polymtl.ca}).}
\and Selime G\"urol\thanks{CERFACS, 42 Avenue Gaspard Coriolis, F-31057 Toulouse Cedex 01, France
  (\email{selime.gurol@cerfacs.fr}).}
\and Alexandre Scotto di Perrotolo\thanks{Universit\'e de Toulouse, INSA-Toulouse, 135 avenue de Rangueil, F-31400 Toulouse, France
  (\email{scottodi@insa-toulouse.fr}).}
\and Xavier Vasseur\thanks{Universit\'e de Toulouse, ISAE-SUPAERO, 10 avenue Edouard Belin, BP~ 54032, F-31055 Toulouse Cedex 4, France
	(\email{xavier.vasseur@isae-supaero.fr}).}
}

\begin{document}

\maketitle

\begin{abstract}
    We propose a general error analysis related to the low-rank approximation of a given real matrix in both the spectral and Frobenius norms. First, we derive deterministic error bounds that hold with some minimal assumptions. Second, we derive error bounds in expectation in the non-standard Gaussian case, assuming a non-trivial mean and a general covariance matrix for the random matrix variable. The proposed analysis generalizes and improves the error bounds for spectral and Frobenius norms proposed by Halko, Martinsson and Tropp. Third, we consider the Randomized Singular Value Decomposition and specialize our error bounds in expectation in this setting. Numerical experiments on an instructional synthetic test case demonstrate the tightness of the new error bounds.
\end{abstract}



\section{Introduction}

Low-rank approximation of large-scale matrices is a key ingredient in numerous applications in data analysis and scientific computing. These applications include   principal component analysis~\cite{RokhlinSzlamEtAl_RandomizedAlgorithmPrincipal_2010}, data compression~\cite{Mahoney_RandomizedAlgorithmsMatrices_2010} and approximation algorithms for partial differential and integral equations~\cite{Hackbusch_HierarchicalMatricesAlgorithms_2015}, to name a few.  Namely, let $A \in \Rnm$ be a rectangular matrix satisfying $n \geq m$ and $Z \in \Rnp$ any full column rank random matrix (with $p \leq \rank(A) \leq m$). If $\pi(Z)$ denotes the orthogonal projection onto the range of $Z$ and $I_n \in \Rnn$ the identity matrix of order $n$, we aim at analyzing the general quantity of interest
\begin{equation} \label{eq:general_error_term}
    \norm{[\I_n - \pi(Z)] A}_{2, F},
\end{equation}
where $\norm{.}_{2,F}$ represents a shortcut for either the spectral norm ($\norm{.}_{2}$) or the Frobenius norm ($\norm{.}_{F}$), respectively.

In our analysis, we will consider the standard Singular Value Decomposition (SVD) of $A$, i.e.,  $A = U \Sigma V\tra $, where $U \in \Rnn$ and $V \in\Rmm$ are orthogonal matrices containing the left and right singular vectors of $A$ respectively and $\Sigma \in \Rnm$ is a diagonal matrix containing the singular values of $A$ denoted as $\sigma_1, \dots, \sigma_m$ (sorted in decreasing order, i.e. $\sigma_1 \geq \sigma_2 \geq \dots \geq \sigma_m \geq 0$). For a target rank $k \in \set{1, \dots, \rank{(A)}}$, which is assumed to be much smaller than $n$, a central decomposition appearing in the error analysis is then given by
\begin{equation} \label{eq:partitioning_svd}
    A =
    \begin{bmatrix}
        U_k & \ubar{U}_k
    \end{bmatrix}
    \begin{bmatrix}
        \Sigma_k & \\
        & \ubar{\Sigma}_k
    \end{bmatrix}
    \begin{bmatrix}
        V_k\tra \\
        \ubar{V}_k\tra
    \end{bmatrix},
\end{equation}
with $U_k \in \Rnk$, $\ubar{U}_k \in \R^{n \times (n-k)}$, $V_k \in \Rmk$, $\ubar{V}_k \in \R^{m \times (m-k)}$, $\Sigma_k \in \Rkk$ and $\ubar{\Sigma}_k \in \R^{(n-k) \times (m-k)}$. We also set $A_k = U_k \Sigma_k V_k\tra $ and  $\ubar{A}_k = \ubar{U}_k \ubar{\Sigma}_k \ubar{V}_k\tra$ so that $A = A_k + \ubar{A}_k.$

The Eckart-Young theorem~\cite{EckartYoung_ApproximationOneMatrix_1936} states that the optimal rank-$k$ approximation of $A$ is given by $A_k = \pi(U_k) A$. In practice, for large-scale applications, computing $U_k$ can be  computationally challenging or too expensive. In this context, randomized algorithms for approximating $U_k$ have become increasingly popular~\cite{HalkoMartinssonEtAl_FindingStructureRandomness_2011, MartinssonTropp_RandomizedNumericalLinear_2020} in the past few years. They have been proved to be easy to implement, computationally efficient and numerically robust. The general idea of randomized subspace iteration methods is to use random sampling to identify a subspace that approximates the range of a given matrix. In this manuscript, we propose a general error analysis related to the low-rank approximation to a given real matrix in the spectral and Frobenius norms.

\subsubsection*{Related research}

The error analysis of randomized algorithms for low-rank approximation to a given matrix has been extensively considered in the literature; see, e.g., the survey papers~\cite{HalkoMartinssonEtAl_FindingStructureRandomness_2011, Mahoney_RandomizedAlgorithmsMatrices_2010, MartinssonTropp_RandomizedNumericalLinear_2020, Woodruff_SketchingToolNumerical_2014} for a general overview. This theoretical analysis takes into account the distribution of the random matrix to derive either error bounds in expectation or tail bounds of the error distribution. The case of standard Gaussian matrices is usually considered, even though alternative theoretical results exist, based on either random column selection matrices~\cite{BoutsidisMahoneyEtAl_ImprovedApproximationAlgorithm_2009} or Subsampled Random Fourier Transform matrices~\cite{HalkoMartinssonEtAl_FindingStructureRandomness_2011}. Halko, Martinsson and Tropp  have developed a reference error analysis in expectation~\cite[Theorems 10.5 and 10.6]{HalkoMartinssonEtAl_FindingStructureRandomness_2011} in the Frobenius and spectral norms for $Z = A \Omega$, $\Omega$ being a standard Gaussian matrix. Later, Gu has refined these error bounds~\cite[Theorem 5.7]{Gu_SubspaceIterationRandomization_2015}.
More recently, Saibaba~\cite{Saibaba_RandomizedSubspaceIteration_2019} has proposed a complementary average case error analysis in terms of principal angles between appropriate subspaces that is available in any unitarily invariant norm.



\subsubsection*{Motivations}

In certain situations, a priori knowledge of $Z$ (or of the corresponding projection subspace) may be available. This naturally arises when considering, e.g., the solution of large-scale nonlinear systems of equations, requiring the solution of a sequence of linear systems of equations. Approximate spectral information based on Ritz or Harmonic Ritz vectors can be easily retrieved to form such a subspace. Exploiting this a priori knowledge is thus of primary interest to design fast, robust and efficient low-rank approximation algorithms. Hence we aim at developing a general theoretical error analysis of randomized algorithms for the low-rank approximation, assuming the existence of a non-trivial mean and of a general covariance matrix for $Z$.

\subsubsection*{Contributions} In this manuscript, for a given target rank, we propose to analyze theoretically the low-rank approximation to a given matrix in both the spectral and Frobenius norms, when $Z$ is drawn from a non-standard Gaussian distribution. Namely, we will first derive in Theorems~\ref{thm:1} and~\ref{thm:2} deterministic error bounds that hold with some minimal assumptions. Second, we will derive error bounds in expectation in the non-standard Gaussian case (with a non-trivial mean value and a general covariance matrix) in Theorems~\ref{thm:3},~\ref{thm:4} and~\ref{thm:5}, respectively. Our analysis simultaneously generalizes and improves the error bounds for spectral and Frobenius norms proposed in~\cite{HalkoMartinssonEtAl_FindingStructureRandomness_2011}. We specialize our error bounds to the Randomized Singular Value Decomposition (RSVD) in Corollaries~\ref{thm:3:rsvd},~\ref{thm:4:rsvd} and~\ref{thm:5:rsvd}, respectively and provide numerical experiments on a synthetic test case that illustrate the tightness of the obtained error bounds.

\subsubsection*{Structure of the manuscript}  Section~\ref{sec:1} introduces specific results useful later in our analysis. Section~\ref{sec:2}  details our error analysis related to the low-rank approximation of the matrix $A$. First, in Section~\ref{sec:2:sub:1}, we derive deterministic error bounds that hold with some minimal assumptions. Then, in Section~\ref{sec:2:sub:2}, we derive error bounds in expectation, with $Z$ drawn following a non-standard Gaussian distribution. In Section~\ref{sec:3},  we specialize our error bounds to the Randomized Singular Value Decomposition. In Section~\ref{sec:4}, we provide detailed numerical illustrations, including a broad comparison with reference error bounds. Conclusions are finally drawn in Section~\ref{sec:5}.

\section{Preliminaries} \label{sec:1}

We first introduce notation used throughout the manuscript and remind specific results. \\




\paragraph{Gaussian matrix}
We use the notation $x \sim \mathcal{N}(\mu, \sigma^2)$ to define a scalar random variable $x$ that follows a Gaussian distribution with mean $\mu$ and standard deviation $\sigma$. A Gaussian matrix $M \in \Rmp$ is a matrix whose columns are independent $m$-dimensional Gaussian vectors, that is $M = [m_1,~\dots,~m_p]$ with $m_i \in \Rm$ for $1 \leq i \leq p$. If $m_i \sim \mathcal{N}(\widehat{m}_i, C)$ where $\widehat{m}_i \in \Rm$ and $C \in \Rmm$, then  we write
\begin{equation} \label{eq:distribution_gaussian_matrix}
    M \sim \mathcal{N}(\widehat{M}, C) \quad \txtwith \quad
    \widehat{M} = [\widehat{m}_1,~\dots,~\widehat{m}_p] \in \Rmp.
\end{equation}
Hence the covariance matrix of the random matrix variable $M$, later explicitly denoted by $\Covmat(M)$, is equal to $C$. An equivalent formulation~\cite[Chapter 5]{Gut_IntermediateCourseProbability_2009} also reads
\begin{equation} \label{eq:gaussian_matrix_relation}
   M = \widehat{M} + \Covmat(M) \half G \quad \txtwith \quad G \sim \mathcal{N}(0, \I_m),
\end{equation}

where $\Covmat(M) \half$ refers either to the positive definite square root of $\Covmat(M)$ if $\Covmat(M)$ is positive definite or to the unique positive semidefinite square root of $\Covmat(M)$ if $\Covmat(M)$ is positive semidefinite~\cite[Theorem 7.2.6]{HornJohnson_MatrixAnalysis_2012}, respectively. \\

\paragraph{Partial ordering on the set of symmetric matrices}
Let $M \in\Rmm, N \in\Rmm$ be two symmetric matrices. The notation $M \lleq N$ means that $N-M$ is positive semidefinite. It defines a partial ordering on the set of symmetric matrices~\cite[Section 7.7]{HornJohnson_MatrixAnalysis_2012}. An  important property is that the partial ordering is preserved under the conjugation rule, i.e.,
\begin{equation} \label{eq:conjugation_rule}
    M \lleq N \implies Q\tra M Q \lleq Q\tra N Q, \quad \forall Q \in \Rmn.
\end{equation}
We note that, in particular, for all $M \in \Rmn$ and $N \in \Rmn$ such that $M\tra M \lleq N\tra N$, we have $\norm{M}_{2,F} \le \norm{N}_{2,F}$. \\

\paragraph{Projection matrices}
Suppose that $M \in \Rnm$ has full column rank with column range space denoted by $\range(M)$. We denote by  $M\pinv$ the left multiplicative inverse of $M$, i.e., the Moore-Penrose inverse of $M$, see, e.g.,~\cite{HornJohnson_MatrixAnalysis_2012}. The orthogonal projection on $\range(M)$ is then given by $\pi(M) = M M\pinv$, in particular, one has $\range(\pi(M)) = \range(M)$. \\

\paragraph{Sherman-Morrison formula} Let $M \in \Rmn$ and $ N \in\Rnm$ such that $\I_n + NM$ is non-singular. Then, $\I_m + MN$ is also non-singular and $(\I_m + MN)\inv = \I_m - M(\I_n + NM)\inv N$ \cite[Section 2.1.4]{GolubVanLoan_MatrixComputations_1996}. \\

\paragraph{Strong submultiplicativity of Frobenius and spectral norms}
The strong submultiplicativity of the spectral and Frobenius norms~\cite[Relation(B.7)]{Higham_FunctionsMatricesTheory_2008} reads
\begin{equation} \label{eq:submultiplicativity}
    \forall M \in \Rnp, \forall  N \in \R^{p \times q}, \forall  Q \in \R^{q \times m},
    \norm{M N Q}_{2,F} \leq \norm{M}_2 ~\norm{N}_{2,F} ~\norm{Q}_2.
\end{equation}

\section{Error bounds for the low-rank approximation of a matrix} \label{sec:2}

Our main objective is to derive error bounds related to an approximation of rank $k$ to $A$ using the orthogonal projection $\pi(Z)$ where $Z \in \Rnp$ and $k \in \{1, \ldots, p\}$, $p \le \rank(A)$. First, in Section~\ref{sec:2:sub:1}, we consider the general case with the minimal assumption that $Z$ is full column rank.  In this case, we are able to derive deterministic error bounds using a systematic approach. Second, in Section~\ref{sec:2:sub:2}, we focus on error bounds that are tractable now from a stochastic point of view, where we assume that the matrix $Z$ corresponds to a general Gaussian matrix. Namely, $Z \in \Rnp$ will be drawn as a Gaussian full column rank matrix of mean $\widehat{Z} \in \Rnp$ and covariance matrix $\Covmat(Z)$, that is, $Z \sim \mathcal{N}(\widehat{Z}, \Covmat(Z))$. In this case, we develop an error analysis in expectation with respect to the random variable $Z$.


\subsection{Deterministic analysis} \label{sec:2:sub:1}

 Without loss of generality, we aim at deriving error bounds for the following quantity
\begin{equation} \label{eq:general_metric}
   \norm{[\I_n - \pi(Z)] A}^2_{2, F}- \norm{[\I_n - \pi(Z)]\ubar{A}_k}^2_{2, F},
\end{equation}

\noindent where  $Z \in \Rnp$ is a full column rank matrix (with $p \le \rank(A)$). Since $\norm{[\I_n - \pi(Z)]\ubar{A}_k}_{2,F}^2 \leq \norm{\ubar{A}_k}_{2,F}^2$, we note that~\eqref{eq:general_metric} is an upper bound  of
\begin{equation} \label{eq:halko_metric}
    \norm{[\I_n - \pi(Z)] A}^2_{2, F}- \norm{\ubar{A}_k}^2_{2, F},
\end{equation}
a quantity which is frequently considered in the analysis of low-rank approximation methods, see, e.g.,~\cite{HalkoMartinssonEtAl_FindingStructureRandomness_2011}. In this sense, our error bounds will naturally cover existing bounds from the literature. The next lemma will be helpful to derive the deterministic error bound for~\eqref{eq:general_metric}.
\begin{lemma}  \label{lem:1}
    Let $A \in \Rnm$ such that $n \geq m$ and $Z \in \Rnp$ a full column rank matrix with $p \leq \rank(A)$. For a given $k \in \{ 1, \dots,p\}$, set $\Omega_k = U_k\tra Z \in \R^{k \times p}$ and $\ubar{\Omega}_k = \ubar{U}_k\tra Z \in \R^{(n-k) \times p}$ such that $\Omega_k$ is full row rank (i.e., $\rank(\Omega_k) = k$).
    Then, one has
    \begin{equation} \label{eq:sine_bound_yd}
        U_k\tra [\I_n - \pi(Z)] U_k \quad \lleq \quad S_k\tra S_k \quad \lleq \quad T_k\tra T_k,
    \end{equation}
    with $T_k = \ubar{\Omega}_k \Omega_k\pinv \in \R^{(n-k) \times k}$ and $S_k = (\I_{n-k} + T_k T_k\tra )\halfinv T_k \in \R^{(n-k) \times k}$.
\end{lemma}
\begin{proof}
    By assumption, $\Omega_k$ has full row rank and therefore has a right multiplicative inverse $\Omega_k\pinv$. Hence $\bar{Z}_k = Z \Omega_k\pinv$ satisfies the two relations
    \begin{equation*}
        U_k\tra \bar{Z}_k = \I_k \quad \txtand \quad \ubar{U}_k\tra \bar{Z}_k = T_k.
    \end{equation*}
    Moreover, we have $\range(\bar{Z}_k) \subset \range(Z)$. Hence, by applying~\cite[Proposition 8.5]{HalkoMartinssonEtAl_FindingStructureRandomness_2011}, one gets $\I_n - \pi(Z) \lleq \I_n - \pi(\bar{Z}_k)$.  By using the conjugation rule (\ref{eq:conjugation_rule}) and the identity $U_k U_k\tra + \ubar{U}_k\ubar{U}_k\tra = \I_n$, we obtain
    \begin{align*}
        U_k\tra [\I_n - \pi(Z)] U_k \lleq U_k\tra [\I_n - \pi(\bar{Z}_k)] U_k & = U_k\tra \left(\I_n - \bar{Z}_k (\bar{Z}_k\tra \bar{Z}_k)\inv \bar{Z}_k\tra \right) U_k, \\
        & = \I_k - U_k\tra \bar{Z}_k (\bar{Z}_k\tra \bar{Z}_k)\inv \bar{Z}_k\tra U_k, \\
        & = \I_k - (\bar{Z}_k\tra \bar{Z}_k)\inv, \\
        & = \I_k - \left(\bar{Z}_k\tra (U_k U_k\tra + \ubar{U}_k \ubar{U}_k\tra)\bar{Z}_k\right)\inv, \\
        & = \I_k - \left(\I_k + T_k\tra T_k \right)\inv, \\
        & = T_k\tra \left(\I_{n-k} + T_k T_k\tra \right)\inv T_k,
    \end{align*}
    where the last equality is obtained using the Sherman-Morrison formula. Then we observe that $T_k\tra \left(\I_{n-k} + T_k T_k\tra \right)\inv T_k = S_k\tra S_k$ and $\left(\I_{n-k} + T_k T_k\tra \right)\inv \lleq \I_{n-k}$. We
    conclude the proof by applying the conjugation rule to deduce $S_k\tra S_k \lleq T_k\tra T_k$.
\end{proof}


\begin{remark}
    In the particular case of $Z= A \Omega$ ($\Omega \in \Rnp$), we have $\Omega_k = \Sigma_k (V_k\tra \Omega)$,  $\ubar{\Omega}_k = \ubar{\Sigma}_k (\ubar{V}_k\tra \Omega)$ and $T_k = \ubar{\Sigma}_k (\ubar{V}_k\tra \Omega) (V_k\tra \Omega)^{\pinv} \Sigma_k^{\inv}$. With the notation of~\cite{HalkoMartinssonEtAl_FindingStructureRandomness_2011}, this gives $\Omega_k := \Sigma_1 \Omega_1$, $\ubar{\Omega}_k := \Sigma_2 \Omega_2$ and $T_k := \Sigma_2 \Omega_2 \Omega_1^{\pinv} \Sigma_1^{\inv}$, respectively. We therefore point out that our notation is related but not directly equivalent to the one employed in~\cite{HalkoMartinssonEtAl_FindingStructureRandomness_2011}.
\end{remark}

\begin{remark}
    We note that the positive singular values of $T_k$ represent the {\it{tangent}} of the canonical angles between $\range(Z \Omega_k\pinv)$ and $\range(U_k)$~\cite[Theorem 3.1 and Remark 3.1]{ZhuKnyazev_AnglesSubspacesTheir_2013}, while the positive singular values of $S_k$ represent the {\it{sine}} of the canonical angles between the same subspaces. Hence we stress this information in the notation and refer the reader to~\cite{Saibaba_RandomizedSubspaceIteration_2019} for a theoretical analysis of low-rank approximation methods in terms of subspace angles.
\end{remark}


\noindent The next theorem introduces unified deterministic error bounds for the quantity of interest~\eqref{eq:general_metric}.

\begin{theorem}(Deterministic error bounds in Frobenius and spectral norms) \label{thm:1}
    Let $A \in \Rnm$ such that $n \geq m$ and $Z \in \Rnp$ a full column rank matrix with $p \leq \rank(A)$. For a given $k \in \{ 1, \dots,p\}$, set $\Omega_k = U_k\tra Z \in \R^{k \times p}$ and $\ubar{\Omega}_k = \ubar{U}_k\tra Z \in \R^{(n-k) \times p}$ such that $\Omega_k$ is full row rank (i.e., $\rank(\Omega_k) = k$). Then, one has
    \begin{equation} \label{deterministic_bound_F_2}
        \norm{[\I_n - \pi(Z)] A}_{2,F}^2 - \norm{[\I_n - \pi(Z)]\ubar{A}_k}_{2,F}^2 \leq \norm{[\I_n - \pi(Z)] A_k}_{2,F}^2 \leq
        \min \set{\norm{S_k}_{2,F}^2 \norm{\Sigma_k}_2^2,~ \norm{T_k \Sigma_k}_{2,F}^2},
    \end{equation}
    where $T_k = \ubar{\Omega}_k \Omega_k\pinv \in \R^{(n-k) \times k}$ and $S_k = (\I_{n-k} + T_k T_k\tra)\halfinv T_k \in \R^{(n-k) \times k}$.
\end{theorem}
\begin{proof}
    Using the identity $AA\tra= A_k A_k\tra +\ubar{A}_k\ubar{A}_k\tra$, we obtain for the spectral norm case
    \begin{align*}
        \norm{[\I_n - \pi(Z)] A}_2^2 & = \norm{[\I_n - \pi(Z)] A A\tra [\I_n - \pi(Z)]}_2, \\
        & = \norm{[\I_n - \pi(Z)] [A_k A_k\tra +\ubar{A}_k\ubar{A}_k\tra] [\I_n - \pi(Z)]}_2.
    \end{align*}
    Hence we obtain
    \begin{align*}
        \norm{[\I_n - \pi(Z)] A}_2^2
        & \leq \norm{[\I_n - \pi(Z)]A_k}_2^2 +  \norm{[\I_n - \pi(Z)]\ubar{A}_k}_2^2.
    \end{align*}
    Thus, by using the unitarily invariance of the spectral norm, we get
    \begin{equation} \label{rel:thm1:1}
        \norm{[\I_n - \pi(Z)] A}_2^2 - \norm{[\I_n - \pi(Z)]\ubar{A}_k}_2^2 \leq \norm{[\I_n - \pi(Z)] U_k \Sigma_k}_2^2.
    \end{equation}
    Moreover, using Lemma~\ref{lem:1} and the conjugation rule, we get
    \begin{equation} \label{rel:thm1:2}
        \Sigma_k\tra U_k\tra  [\I_n - \pi(Z)] U_k \Sigma_k \lleq \Sigma_k\tra S_k\tra S_k \Sigma_k \lleq \Sigma_k T_k\tra T_k \Sigma_k.
    \end{equation}
    Hence, since $[\I_n - \pi(Z)]^2=\I_n - \pi(Z)$, we obtain
    \begin{equation*}
        \norm{[\I_n - \pi(Z)] U_k \Sigma_k}^2_2 = \norm{\Sigma_k\tra U_k\tra [\I_n - \pi(Z)] U_k \Sigma_k}_2
        \leq \norm{\Sigma_k\tra S_k\tra S_k \Sigma_k}_2
        \leq \norm{\Sigma_k\tra T_k\tra T_k \Sigma_k}_2,
    \end{equation*}
    or equivalently,
    \begin{equation*}
        \norm{[\I_n - \pi(Z)] U_k \Sigma_k}^2_2 \leq \norm{S_k \Sigma_k}_2^2 \leq \norm{T_k \Sigma_k}_2^2.
    \end{equation*}
    Using the strong submultiplicativity of the spectral norm, the latter inequality implies
    \begin{equation} \label{rel:thm1:3}
        \norm{[\I_n - \pi(Z)] U_k \Sigma_k}^2_2 \leq \min \set{ \norm{S_k}_2^2 \norm{\Sigma_k}_2^2,~ \norm{T_k \Sigma_k}_2^2}.
    \end{equation}
    Combining~\eqref{rel:thm1:1} and~\eqref{rel:thm1:3} completes the proof for the spectral norm case.

    \noindent For the Frobenius norm, similar arguments are used. In fact, we have
    \begin{align*}
        \norm{[\I_n - \pi(Z)] A}_F^2 & = \tr\left( [\I_n - \pi(Z)] A A\tra [\I_n - \pi(Z)]\right), \\
        & = \tr\left( [\I_n - \pi(Z)] [A_k A_k\tra + \ubar{A}_k\ubar{A}_k\tra] [\I_n - \pi(Z)] \right), \\
        & = \tr\left( [\I_n - \pi(Z)] A_k A_k\tra [\I_n - \pi(Z)] \right) + \tr\left( [\I_n - \pi(Z)]\ubar{A}_k\ubar{A}_k\tra [\I_n - \pi(Z)] \right), \\
        & = \norm{[\I_n - \pi(Z)]A_k}_F^2 + \norm{[\I_n - \pi(Z)] \ubar{A}_k}_F^2.
    \end{align*}
     Thus,
    \begin{equation} \label{rel:thm1:4}
        \norm{[\I_n - \pi(Z)] A}_F^2 - \norm{[\I_n - \pi(Z)]\ubar{A}_k}_F^2 = \norm{[\I_n - \pi(Z)] U_k \Sigma_k}_F^2.
    \end{equation}
    Using~\eqref{rel:thm1:2}, we obtain
    \begin{equation*}
        \norm{[\I_n - \pi(Z)] U_k \Sigma_k}^2_F = \tr\left( \Sigma_k\tra U_k\tra [\I_n - \pi(Z)] U_k \Sigma_k \right)
        \leq \tr\left( \Sigma_k\tra S_k\tra S_k \Sigma_k \right)
        \leq \tr\left( \Sigma_k\tra T_k\tra T_k \Sigma_k \right).
    \end{equation*}
    Hence we obtain
    \begin{equation} \label{rel:thm1:5}
        \norm{[\I_n - \pi(Z)] U_k \Sigma_k}^2_F \leq \min \set{ \norm{S_k}_F^2 \norm{\Sigma_k}_2^2,~ \norm{T_k \Sigma_k}_F^2}.
    \end{equation}
    Combining relations~\eqref{rel:thm1:4} and~\eqref{rel:thm1:5} completes the proof for the Frobenius norm case.
\end{proof}


\begin{remark}
    We have shown in Theorem~\ref{thm:1} that
    \begin{equation}
        \norm{[\I_n - \pi(Z)] A}_{2,F}^2 - \norm{[\I_n - \pi(Z)]\ubar{A}_k}_{2,F}^2 \leq
        \norm{S_k \Sigma_k}_{2,F}^2 \leq \norm{T_k \Sigma_k}_{2,F}^2.
    \end{equation}
    Hence,  $\norm{S_k \Sigma_k}_{2,F}^2$ definitively represents a sharper upper bound in the deterministic case. Nevertheless we have decided to use the submultiplicativity to bound $\norm{S_k \Sigma_k}_{2,F}^2$, since only this formulation authorizes a possible treatment in the stochastic setting as detailed in Section~\ref{sec:2:sub:2}.
\end{remark}

\begin{remark}
    In the particular case of $Z = A \Omega$ ($\Omega \in \Rnp$), with the notation of~\cite{HalkoMartinssonEtAl_FindingStructureRandomness_2011}, we have $\norm{T_k \Sigma_k}_{2,F} := \norm{\Sigma_2 \Omega_2 \Omega_1^{\pinv}}_{2,F}$, a quantity which precisely appears in the reference upper bound~\cite[Theorem 9.1]{HalkoMartinssonEtAl_FindingStructureRandomness_2011}. We note that our error bound~\eqref{deterministic_bound_F_2} is tighter whenever $\norm{S_k}_{2,F}^2 \norm{\Sigma_k}_2^2 <  \norm{T_k \Sigma_k}_{2,F}^2$. Hence Theorem~\ref{thm:1} either recovers or improves the reference error bound~\cite[Theorem 9.1]{HalkoMartinssonEtAl_FindingStructureRandomness_2011} in this setting.
\end{remark}

\begin{remark}
    Using $\norm{[\I_n - \pi(Z)]\ubar{A}_k}_{2,F}^2 \leq \norm{\ubar{A}_k}_{2,F}^2$ and  $\sqrt{a^2+b^2} \leq a+b$ for any real positive scalars $a$ and $b$, we note that Theorem~\ref{thm:1} implies~\cite[Theorem 3.3]{DrineasIpsen_LowRankMatrixApproximations_2019} when stated in the spectral and Frobenius norms. As~\cite[Theorem 3.3]{DrineasIpsen_LowRankMatrixApproximations_2019}, Theorem~\ref{thm:1} can be extended to the case of any Schatten-$p$ class of norms.
\end{remark}


If we target to bound $\norm{[\I_n - \pi(Z)] A}_2^2 - \norm{\ubar{A}_k}^2_2$ instead of $\norm{[\I_n - \pi(Z)] A}_2^2 - \norm{[\I_n - \pi(Z)]\ubar{A}_k}_2^2$, we are able to improve the result given in Theorem~\ref{thm:1}. This is detailed next.

\begin{theorem}(Improved deterministic error bound in spectral norm) \label{thm:2}
    Let $A \in \Rnm$ such that $n \geq m$ and $Z \in \Rnp$ a full column rank matrix with $p \leq \rank(A)$. For a given $k \in \{ 1, \dots,p\}$, set $\Omega_k = U_k\tra Z \in \R^{k \times p}$ and $\ubar{\Omega}_k = \ubar{U}_k\tra Z \in \R^{(n-k) \times p}$ such that $\Omega_k$ is full row rank (i.e., $\rank(\Omega_k) = k$).    Then, one has
    \begin{equation} \label{deterministic_bound_2_enhanced}
        \norm{[\I_n - \pi(Z)] A}_2^2 - \norm{\ubar{A}_k}^2_2 \leq
        \min \set{ \norm{S_k}_2^2 \norm{\widehat{\Sigma}_k}_2^2,~\norm{T_k \widehat{\Sigma}_k}_2^2 },
    \end{equation}
    where $\widehat{\Sigma}_k = \left(\Sigma_k^2 - \sigma_{k+1}^2 \I_k\right)\half \in \R^{k \times k}$,  $T_k = \ubar{\Omega}_k \Omega_k\pinv \in \R^{(n-k) \times k}$ and $S_k = (\I_{n-k} + T_k T_k\tra)\halfinv T_k \in \R^{(n-k) \times k}$.
\end{theorem}
\begin{proof}
    First, we note that $\norm{\ubar{A}_k}^2_2=\sigma_{k+1}^2$. Then, by definition of $\sigma_{k+1}$ and $\ubar{\Sigma}_k$, one has  $\ubar{\Sigma}_k \ubar{\Sigma}_k\tra  \lleq \sigma_{k+1}^2 \I_{n-k}$. Thus, we obtain
    \begin{align*}
        A A\tra  & \lleq U_k \Sigma_k^2 U_k\tra + \sigma_{k+1}^2 \ubar{U}_k \ubar{U}_k\tra,  \\
        & \lleq U_k \Sigma_k^2 U_k\tra + \sigma_{k+1}^2 (\I_n - U_k U_k\tra ),\\
        & \lleq U_k \widehat{\Sigma}_k^2 U_k\tra + \sigma_{k+1}^2 \I_n.
    \end{align*}
    Hence,
    \begin{align*}
        \norm{[\I_n - \pi(Z)] A}_2^2 & = \norm{[\I_n - \pi(Z)] A A\tra [\I_n - \pi(Z)]}_{2}, \\
        & \leq \sigma_{k+1}^2 \norm{\I_n - \pi(Z)}_2 + \norm{[\I_n - \pi(Z)] U_k \widehat{\Sigma}_k^2 U_k\tra [\I_n - \pi(Z)]}_{2}, \\
        & \leq \norm{\ubar{A}_k}^2_2 + \norm{[\I_n - \pi(Z)] U_k \widehat{\Sigma}_k}_2^2.
    \end{align*}
    Then, using~\eqref{rel:thm1:3} with $\widehat{\Sigma}_k$ instead of $\Sigma_k$, the rest of the proof
    follows straightforwardly.
\end{proof}

By definition of $\widehat{\Sigma}_k$, one has $\widehat{\Sigma}_k \lleq \Sigma_k$. Thus we deduce that if we target to bound $\norm{[\I_n - \pi(Z)] A}_2^2 - \norm{\ubar{A}_k}^2_2$, then the error bound~\eqref{deterministic_bound_2_enhanced} is tighter than~\eqref{deterministic_bound_F_2}. To the best of our knowledge, Theorem~\ref{thm:2} is new and provides an improved error bound in the spectral norm for a fairly general choice of $Z$.

Finally, we emphasize that both Theorems~\ref{thm:1} and~\ref{thm:2} will play a key role for deriving new  improved error bounds in expectation for the low-rank approximation to a given matrix. This is detailed next.


\subsection{Stochastic analysis} \label{sec:2:sub:2}

We now provide error bounds in expectation and consider the case where  $Z \in \Rnp$ is drawn as a  Gaussian matrix of mean $\widehat{Z} \in \Rnp$ and covariance matrix $\Covmat(Z)$, that is, $Z \sim \mathcal{N}(\widehat{Z}, \Covmat(Z))$, with $2 < p \le
\min\set{\rank(A), \, \rank(\Covmat(Z))}$. For $k \in \{1, \ldots, p-2\}$, we  define $\Omega_k \in \R^{k \times p}$ and $\ubar{\Omega}_k \in \R^{(n-k) \times p}$ as  $\Omega_k = U_k\tra (Z- \widehat{Z})$ and $\ubar{\Omega}_k = \ubar{U}_k\tra (Z- \widehat{Z})$, respectively. The condition $p \le \rank{(\Covmat(Z))}$ ensures that $Z- \widehat{Z}$ has full column rank with probability one~\cite{FengZhang_RankRandomMatrix_2007}. We assume that $\Omega_k$ has full row rank in this section. \\

This general approach offers several advantages but rises additional technical difficulties. In particular, $\Omega_k$ and $\ubar{\Omega}_k$ are now stochastically dependent on the distribution law of $Z$. Thus, before stating our main results in Section~\ref{sec:2:sub:2:sub:2}, we provide in Section~\ref{sec:2:sub:2:sub:1} preparatory lemmas that extend existing probabilistic results to the non-standard Gaussian case. In this section, we consider the following block partitioning of the projected covariance matrix $U\tra \Covmat(Z)\, U$\begin{equation*}
 U\tra \Covmat(Z)\, U=   \begin{bmatrix}
        ~U_k\tra~ \\
        ~\ubar{U}_k\tra~
    \end{bmatrix} \Covmat(Z)
    \begin{bmatrix}
        ~U_k & \ubar{U}_k~
    \end{bmatrix} =
    \begin{bmatrix}
        \Covmat_k(Z) & \Covmat_{\perp,k}(Z)\tra  \\
        \Covmat_{\perp,k}(Z) & \uwidebar{\Covmat}_k(Z)
    \end{bmatrix},
\end{equation*}
with
\begin{equation} \label{eq:partitioning_covariance_matrix}
    \begin{aligned}
        \Covmat_k(Z) & = U_k\tra \Covmat(Z) U_k \in \Rkk, \\
        \Covmat_{\perp,k}(Z) & = \ubar{U}_k\tra \Covmat(Z) U_k\in \real^{(n-k) \times k},  \\
        \uwidebar{\Covmat}_k(Z) & = \ubar{U}_k\tra \Covmat(Z) \ubar{U}_k\in \real^{(n-k)\times (n-k)}.
    \end{aligned}
\end{equation}


\subsubsection{Preparatory lemmas} \label{sec:2:sub:2:sub:1}

Given that $\Omega_k = U_k\tra (Z- \widehat{Z})$ and $\ubar{\Omega}_k = \ubar{U}_k\tra  (Z- \widehat{Z})$, then by using elementary properties of Gaussian vectors, one gets that $\Omega_k \sim \mathcal{N}(0, \Covmat_k(Z))$ and $\ubar{\Omega}_k \sim \mathcal{N}(0, \uwidebar{\Covmat}_k(Z))$. We note that, although $\Omega_k $ and $\ubar{\Omega}_k$ are centered Gaussian matrices, the conditional law of $\ubar{\Omega}_k$ with respect to $\Omega_k$ follows a Gaussian distribution that is not necessarily centered~\cite[Theorem 1.2.11]{Muirhead_AspectsMultivariateStatistical_1982}. We therefore adapt this result to our setting in the next lemma.

\begin{lemma} \label{lem:3}
    Let $A \in \Rnm$ such that $n \geq m$ and $Z \in \Rnp$ a Gaussian full column rank matrix of mean $\widehat{Z} \in \Rnp$ and covariance matrix $\Covmat(Z)$, that is, $Z \sim \mathcal{N}(\widehat{Z}, \Covmat(Z))$ satisfying $2 < p \leq \min\set{\rank(A), \, \rank(\Covmat(Z))}$. For a given $k \in \set{1, \dots, p-2}$, set $\Omega_k = U_k\tra (Z - \widehat{Z}) \in \R^{k \times p}$, $\ubar{\Omega}_k = \ubar{U}_k\tra  (Z - \widehat{Z}) \in \R^{(n-k) \times p}$ such that $\Omega_k$ is full row rank (i.e. $\rank(\Omega_k) = k$). If the projected covariance matrix $\Covmat_k(Z)$ is non-singular, then the random matrix $\ubar{\Omega}_k$ conditioned by $\Omega_k$ follows a Gaussian distribution of mean
    \begin{equation} \label{eq:conditional_dsitribution:mean}
        \expect{\ubar{\Omega}_k \mid \Omega_k} = \Covmat_{\perp,k}(Z) [\Covmat_k(Z)]\inv \Omega_k,
    \end{equation}
    and of covariance matrix given by
    \begin{equation} \label{eq:conditional_dsitribution:cov}
        \Covmat\left(\ubar{\Omega}_k \mid \Omega_k\right) =
        \uwidebar{\Covmat}_k(Z) - \Covmat_{\perp,k}(Z) \left[\Covmat_k(Z)\right]\inv \Covmat_{\perp,k}(Z)\tra,
    \end{equation}
    where $\Covmat_k(Z) = U_k\tra \Covmat(Z) U_k$, $\Covmat_{\perp,k}(Z) = \ubar{U}_k\tra \Covmat(Z) U_k$ and $\uwidebar{\Covmat}_k(Z) = \ubar{U}_k\tra \Covmat(Z) \ubar{U}_k$.
\end{lemma}


\begin{remark}
When $\Covmat{(Z)} = A A \tra$, we note that $\Covmat_{\perp,k}(Z)~=~0$, which yields \linebreak $\expect{ \ubar{\Omega}_k \mid \Omega_k} = 0$ and   $\Covmat\left(\ubar{\Omega}_k \mid \Omega_k\right) = \uwidebar{\Covmat}_k(Z)$.
\end{remark}



The next two lemmas aim at proposing key results about Gaussian matrices in the non-standard case. In particular, we extend~\cite[Propositions 10.1 and 10.2]{HalkoMartinssonEtAl_FindingStructureRandomness_2011} to the case of a non-standard Gaussian matrix with general covariance matrix and potentially nonzero mean term.
\begin{lemma} \label{lem:4}
    Let $N \in \Rpp$ be a given matrix and $M \in \R^{k \times p}$ be a Gaussian matrix such that $M \sim \mathcal{N}(\widehat{M}, \Covmat(M))$ with mean $\widehat{M} \in \R^{k \times p}$ and covariance matrix $\Covmat(M) \in \R^{k \times k}$. Then
    \begin{equation} \label{eq:expectation_product_norm:2}
        \expect{\norm{M N}_2} \leq
        \norm{\widehat{M}N}_2 + \norm{\Covmat(M)\half}_2 \norm{N}_F + \norm{\Covmat(M)\half}_F \norm{N}_2,
    \end{equation}
    and
    \begin{equation} \label{eq:expectation_product_norm:F}
        \expect{\norm{M N}_F^2} = \norm{\widehat{M}N}_F^2 + \norm{\Covmat(M) \half}_F^2 \norm{N}_F^2.
    \end{equation}
\end{lemma}
\begin{proof}
    Using the definition of a non-standard Gaussian matrix~\eqref{eq:gaussian_matrix_relation}, we have $M = \widehat{M} + \Covmat(M) \half G$ with $G \in \R^{k \times p }$ a standard Gaussian matrix (i.e. $G \sim \mathcal{N}(0, I_k)$). Thus, by applying the triangle inequality in the spectral norm, we get
    \begin{equation*}
        \norm{M N}_{2} = \norm{(\widehat{M} + \Covmat(M)\half G) N}_{2} \leq \norm{\widehat{M}N}_{2} + \norm{\Covmat(M)\half G N}_{2}.
    \end{equation*}
    Then, by applying~\cite[Proposition 10.1]{HalkoMartinssonEtAl_FindingStructureRandomness_2011} to obtain an upper bound of $\expect{\norm{\Covmat(M)\half G N}_{2}}$, we deduce~\eqref{eq:expectation_product_norm:2}.
    In the Frobenius norm, we have
    \begin{equation*}
        \norm{M N}_F^2 = \norm{\widehat{M}N + \Covmat(M)\half G N}_F^2 = \norm{\widehat{M}N}_F^2 + \norm{\Covmat(M)\half G N}_F^2 + 2\tr (N\tra \widehat{M}\tra \Covmat(M)\half G N).
    \end{equation*}
    Taking the expectation and using~\cite[Proposition 10.1]{HalkoMartinssonEtAl_FindingStructureRandomness_2011}, we get
    \begin{equation*}
        \expect{\norm{M N}_F^2} = \norm{\widehat{M}N}_F^2 + \norm{\Covmat(M)\half}_F^2 \norm{N}_F^2 + 2~\expect{\tr (N\tra \widehat{M}\tra \Covmat(M)\half G N)}.
    \end{equation*}
    By using the linearity of the expectation and the fact that $\expect{G} = 0$, we remark that
    \begin{equation*}
        \expect{\tr (N\tra \widehat{M}\tra \Covmat(M) \half G N)}= \tr \left(N\tra \widehat{M}\tra \Covmat(M)\half \expect{G} N\right) = 0.
    \end{equation*}
    This concludes the proof.
\end{proof}

In the next lemma, we show how to bound (in expectation) the quantity $\norm{M\pinv N}_{2}$ or $\norm{M\pinv N}^2_{F}$, where $N \in \Rkk$ is a given fixed matrix and $M \in \R^{k \times p}$ follows a centered Gaussian distribution.

\begin{lemma} \label{lem:5}
    For a fixed integer $k>1$, let $N \in \Rkk$ be a given matrix and $M \in \R^{k \times p}$ (with $p > k + 1$)  be a centered  Gaussian matrix $M \sim \mathcal{N}(0, \Covmat(M))$. If the covariance matrix $\Covmat(M)$ is non-singular, then
    \begin{equation*} \label{eq:expectation_pseudoinverse_norm}
       \expect{\norm{M\pinv N}_F^2} = \frac{\norm{(N\tra [\Covmat(M)]\inv N)\half}_F}{p - k - 1}
       \quad \txtand \quad
       \expect{\norm{M\pinv N}_2} \leq \frac{e \sqrt{p}}{p-k}\sqrt{\norm{N\tra [\Covmat(M)]\inv N}_2},
    \end{equation*}
    where $e$ denotes exponential of $1$, i.e., $e=\exp(1)$.
\end{lemma}
\begin{proof}
    Since $M$ is a random matrix with $p$ independent columns, where each column is a multivariate Gaussian distribution with zero mean and covariance matrix $\Covmat(M)$, the matrix $M M\tra$ follows a Wishart distribution of the form $\mathcal{W}_k(p, \Covmat(M))$\cite[Definition 3.1.3]{Muirhead_AspectsMultivariateStatistical_1982}.  One has also $\norm{M\pinv N}_F^2 = \tr(N \tra [(M\pinv)\tra M \pinv] N) = \tr(N\tra [MM\tra]\inv N)$, where the second equality holds with probability one since $p > k+1$. In fact, if $\Covmat(M)$ is non-singular, the matrix $M M\tra$ is non-singular almost surely, see~\cite[Theorem 3.1.4]{Muirhead_AspectsMultivariateStatistical_1982}. In this case, according to~\cite[Theorem 3.2.12]{Muirhead_AspectsMultivariateStatistical_1982} for $p > k + 1$, one has, for any matrix $N \in \Rkk$,
    \begin{equation*}
        \expect{N\tra [MM\tra]\inv N} = N\tra \expect{[MM\tra]\inv} N = \displaystyle \frac{N\tra \Covmat(M) \inv N}{p - k - 1}.
    \end{equation*}
     Hence, $$\expect{\tr(N\tra [M M\tra]\inv N)} = \displaystyle \frac{\tr(N\tra \Covmat(M)\inv N)}{p - k - 1}= \displaystyle \frac{\norm{(N\tra \Covmat(M)\inv N)\half}_F}{p - k - 1},$$ which concludes the proof for the Frobenius norm.
    In the spectral norm, using~\eqref{eq:gaussian_matrix_relation}, we have $M =\Covmat(M) \half G$ with $G \in \R^{k \times p}$ following a standard Gaussian distribution. If $\Covmat(M)$ is non-singular, then one has $M\pinv = G\pinv \Covmat(M)\halfinv$  and thus, for any matrix $N \in \Rkk$, we have
    \begin{equation*}
        \norm{M\pinv N}_2 = \norm{G\pinv \Covmat(M)\halfinv N}_2 \leq \norm{G\pinv}_2 \norm{\Covmat(M)\halfinv N}_2 = \norm{G\pinv}_2 \sqrt{\norm{N\tra \Covmat(M)\inv N}_2}.
    \end{equation*}
    Then, we take the expectation and apply~\cite[Proposition 10.2]{HalkoMartinssonEtAl_FindingStructureRandomness_2011} to bound the quantity $\expect{\norm{G\pinv}_2}$, which concludes the proof.
\end{proof}

We now introduce a final lemma based on Lemmas~\ref{lem:3},~\ref{lem:4} and~\ref{lem:5}, which is central in the derivation of our error bounds in expectation.

\begin{lemma} \label{lem:6}
    Let $A \in \Rnm$ such that $n \geq m$ and $Z \in \Rnp$ a Gaussian matrix of mean $\widehat{Z} \in \Rnp$ and covariance matrix $\Covmat(Z)$, that is, $Z \sim \mathcal{N}(\widehat{Z}, \Covmat(Z))$, satisfying $2 < p \leq \min\set{\rank(A), \, \rank(\Covmat(Z))}$. For a given $k \in \set{1, \dots, p-2}$, set $\Omega_k = U_k\tra (Z - \widehat{Z}) \in \R^{k \times p}$, $\ubar{\Omega}_k = \ubar{U}_k\tra  (Z - \widehat{Z}) \in \R^{(n-k) \times p}$ such that $\Omega_k$ is full row rank (i.e. $\rank(\Omega_k) = k$) and $T_k = \ubar{\Omega}_k \Omega_k\pinv$.

    If the covariance matrix $\Covmat_k(Z)$ is non-singular, then, for any matrix $N \in \Rkk$, one has
    \begin{equation} \label{eq:expectation_norm_omega_k:2}
        \expect{\norm{T_k N}_2} \leq \CtotSp(\Covmat(Z), N) := \CdepSp(\Covmat(Z), N) + \CSp(\Covmat(Z), N),
    \end{equation}
    and
    \begin{equation} \label{eq:expectation_norm_omega_k:F}
        \expect{\norm{T_k N}_F^2} =  \CtotF(\Covmat(Z), N) := \CdepF(\Covmat(Z), N)^2 + \CF(\Covmat(Z), N)^2.
    \end{equation}

    The (positive) constants are defined as
    \begin{equation} \label{eq:expectation_norm_omega_k:C_dep}
        \CdepSpF(\Covmat(Z), N) = \norm{\Covmat_{\perp,k}(Z) [\Covmat_k(Z)]\inv N}_{2,F},
    \end{equation}

    \begin{equation} \label{eq:expectation_norm_omega_k:C2}
        \begin{aligned}
            \CSp(\Covmat(Z), N) & = \frac{\norm{\Covmat\left(\ubar{\Omega}_k \mid \Omega_k\right)\half}_2 \norm{(N\tra [\Covmat_k(Z)]\inv N)\half}_F}{\sqrt{p-k-1}}  \\
            & \quad + \frac{e \sqrt{p}}{p-k}\norm{\Covmat\left(\ubar{\Omega}_k \mid \Omega_k\right)\half}_F \norm{(N\tra [\Covmat_k(Z)]\inv N)\half}_2,
        \end{aligned}
    \end{equation}
    and
    \begin{eqnarray} \label{eq:expectation_norm_omega_k:CF}
        \CF(\Covmat(Z), N) = \frac{\norm{\Covmat\left(\ubar{\Omega}_k \mid \Omega_k\right)\half}_F\norm{(N\tra [\Covmat_k(Z)]\inv N)\half}_F}{\sqrt{p-k-1}},
    \end{eqnarray}
    where {$\Covmat_k(Z) = U_k\tra \Covmat(Z) U_k$, $\Covmat_{\perp,k}(Z)= \ubar{U}_k\tra \Covmat(Z) U_k$,  $\uwidebar{\Covmat}_k(Z) = \ubar{U}_k\tra \Covmat(Z) \ubar{U}_k$ and $\Covmat\left(\ubar{\Omega}_k \mid \Omega_k\right) = \uwidebar{\Covmat}_k(Z) - \Covmat_{\perp,k}(Z) \left[\Covmat_k(Z)\right]\inv \Covmat_{\perp,k}(Z)\tra$}.
\end{lemma}
\begin{proof}
    In the spectral norm, Lemma~\ref{lem:4} gives, for any matrix $N  \in \Rkk$,
    \begin{equation} \label{lem:6:eq:1}
        \begin{aligned}
            \expect{\norm{\ubar{\Omega}_k \Omega_k\pinv N}_2 \mid \Omega_k} & \leq
            \norm{\expect{\ubar{\Omega}_k \mid \Omega_k} \Omega_k\pinv N}_2 + \norm{\Covmat\left(\ubar{\Omega}_k \mid \Omega_k\right)\half}_2 \norm{\Omega_k\pinv N}_F \\
            & \quad + \norm{\Covmat\left(\ubar{\Omega}_k \mid \Omega_k\right)\half}_F \norm{\Omega_k\pinv N}_2.
        \end{aligned}
    \end{equation}
    From Lemma~\ref{lem:3}, one has for any matrix $N \in \Rkk$,
    \begin{equation*}
        \expect{ \ubar{\Omega}_k \mid \Omega_k}  \Omega_k\pinv N = \Covmat_{\perp,k}(Z) [\Covmat_k(Z)]\inv \Omega_k \Omega_k\pinv N = \Covmat_{\perp,k}(Z) [\Covmat_k(Z)]\inv N,
    \end{equation*}
    which is a deterministic constant. Hence, taking the total expectation of~\eqref{lem:6:eq:1} leads to
    \begin{align*}
        \expect{\norm{T_k N}_2} & = \expect{\norm{\ubar{\Omega}_k \Omega_k\pinv N}_2} =  \expect{\expect{\norm{\Omega_\perp \Omega_k\pinv N}_2 \mid \Omega_k}}, \\
        & \leq \norm{\Covmat_{\perp,k}(Z) [\Covmat_k(Z)]\inv N}_2  +\norm{\Covmat\left(\ubar{\Omega}_k \mid \Omega_k\right)\half}_2 \expect{\norm{\Omega_k\pinv N}_F} \\
        & \quad + \norm{\Covmat\left(\ubar{\Omega}_k \mid \Omega_k\right)\half}_F \expect{\norm{\Omega_k\pinv N}_2}.
    \end{align*}
    Finally, by using Lemma~\ref{lem:5}, one has
    \begin{equation*}
        \expect{\norm{\Omega_k\pinv N}_F} \leq
        \expect{\norm{\Omega_k\pinv N}_F^2}\half =
        \frac{\norm{(N\tra [\Covmat_k(Z)]\inv N)\half}\half_F}{\sqrt{p-k-1}}
    \end{equation*}
    and
    \begin{equation*}
        \expect{\norm{\Omega_k\pinv N}_2} \leq ~\frac{e \sqrt{p}}{p-k} \norm{(N\tra [\Covmat_k(Z)]\inv N)\half}_2.
    \end{equation*}
    This concludes the proof in the spectral norm case. The proof related to the Frobenius norm can be derived similarly.
\end{proof}

\begin{remark}
We note that both $\CdepSp(\Covmat(Z), N)$ and $\CdepF(\Covmat(Z), N)$ do depend on \linebreak $\Covmat_{\perp,k}(Z)$, which is the term related to the statistical dependence between $\Omega_k$ and $\ubar{\Omega}_k$. Those two terms cancel out whenever $\Omega_k$ and $\ubar{\Omega}_k$ are independent. By contrast, $\CSp(\Covmat(Z), N)$ and $\CF(\Covmat(Z), N)$ will not cancel out if $\Omega_k$ and $\ubar{\Omega}_k$ are independent, but do approach zero as the number of samples $p$ increases.
\end{remark}


\subsubsection{Error bounds in expectation} \label{sec:2:sub:2:sub:2}

We are now able to provide a first key result. Given $Z \sim \mathcal{N}(\widehat{Z}, \Covmat(Z))$, we aim at bounding (in expectation) the quantity $\norm{S_k}_{2,F}$, where $S_k = (\I_{n-k} + T_k T_k\tra )\halfinv T_k$ and $T_k = \ubar{\Omega}_k \Omega_k\pinv$. We recall that the positive singular values of $S_k$ represent the sine of the canonical angles between $\range{(Z\Omega_k\pinv)}$ and $\range{(U_k)}$. This will make our proposed error bounds accurate in the presence of large canonical angles,  compared to~\cite{HalkoMartinssonEtAl_FindingStructureRandomness_2011} where the analysis is based on the tangent of the canonical angles (via the matrix $T_k$). We highlight this result in the next proposition.

\begin{proposition} \label{prop:1}
    Let $A \in \Rnm$ such that $n \geq m$ and $Z \in \Rnp$ a Gaussian matrix of mean $\widehat{Z} \in \Rnp$ and covariance matrix $\Covmat(Z)$, that is, $Z \sim \mathcal{N}(\widehat{Z}, \Covmat(Z))$, satisfying $2 < p \leq \min\set{\rank(A), \, \rank(\Covmat(Z))}$. For a given $k \in \set{1, \dots, p-2}$, set $\Omega_k = U_k\tra (Z - \widehat{Z}) \in \R^{k \times p}$, $\ubar{\Omega}_k = \ubar{U}_k\tra  (Z - \widehat{Z}) \in \R^{(n-k) \times p}$ such that $\Omega_k$ is full row rank (i.e. $\rank(\Omega_k) = k$), $T_k = \ubar{\Omega}_k \Omega_k\pinv \in \R^{(n-k) \times k}$ and $S_k = (\I_{n-k} + T_k T_k\tra )\halfinv T_k \in \R^{(n-k) \times k}$. Let  $\varphi : x \mapsto x / \sqrt{1 + x^2}$ for $x \geq 0$. We have
    \begin{equation*} \label{eq:expectation_norm_S}
        \expect{\norm{S_k}_2} \leq \varphi \left( \CtotSp(\Covmat(Z), \I_k) \right)
        \quad \txtand \quad
        \expect{\norm{S_k}_F} \leq \sqrt{k} \varphi \left( \dfrac{1}{\sqrt{k}} ~\sqrt{\CtotF(\Covmat(Z), \I_k)} \right),
    \end{equation*}
    where the positive constants $\CtotSp(\Covmat(Z), \I_k)$ and $\CtotF(\Covmat(Z), \I_k)$ are given in Lemma~\ref{lem:6} (with $N = \I_k$).
\end{proposition}
\begin{proof}
    We define $\sigma_1(T_k) \geq \dots \geq \sigma_{k}(T_k)$ (resp. $\sigma_1(S_k) \geq \dots \geq \sigma_{k}(S_k)$) as the positive singular values of $T_k$ (resp. $S_k$). Since $\varphi$ is an increasing map, one has\footnote{Since the positive singular values of $T_k$ represent the {\it{tangent}} of the canonical angles between $\range(Z \Omega_k\pinv)$ and $\range(U_k)$, relation~\eqref{eq:relation_sigma_sin_tan} shows that the positive singular values of $S_k$ are the sine of the canonical angles between the same subspaces.}
    \begin{equation} \label{eq:relation_sigma_sin_tan}
        \sigma_i(S_k)= \frac{\sigma_i(T_k)}{\sqrt{1 + \sigma_i(T_k)^2}} = \varphi\left(\sigma_i(T_k)\right ), \quad i \in \set{1, \dots, k}.
    \end{equation}
    Then, we obtain
    \begin{equation*} \label{eq:norm_S:spec}
        \norm{S_k}_2 = \sigma_1(S_k) = \varphi\left(\sigma_1(T_k)\right) = \varphi(\norm{T_k}_2).
    \end{equation*}
    Taking the expectation then leads to
    \begin{equation*}
        \expect{\norm{S_k}_2}
        \leq \varphi\left( \expect{\norm{T_k}_2} \right)
        \leq \varphi\left( \CtotSp(\Covmat(Z), \I_k) \right),
    \end{equation*}
    where the first inequality relies on both the concavity of $\varphi$ and Jensen's inequality, while the second one uses~\eqref{eq:expectation_norm_omega_k:2} of Lemma~\ref{lem:6} and the fact that $\varphi$ is an increasing map.
    In the Frobenius norm, one has
    \begin{equation*}
        \norm{S_k}_F^2 = \tr(S_k\tra S_k) = \sum_{i=1}^{k} \sigma_i(S_k)^2 = \sum_{i=1}^{k} \frac{\sigma_i(T_k)^2}{1 + \sigma_i(T_k)^2}.
    \end{equation*}
    We note that the scalar map $\psi: x \to \displaystyle \frac{x}{1+x}$ is concave (for all $x\ge 0$). Thus, the Jensen's inequality yields






    \begin{equation*} \label{eq:norm_S:fro}
        \frac{1}{k} \norm{S_k}_F^2
        = \frac{1}{k} \sum_{i=1}^{k} \psi \left( \sigma_i(T_k)^2 \right)
        \leq \psi \left(\frac{1}{k} \sum_{i=1}^{k} \sigma_i(T_k)^2 \right)
        = \left[\varphi \left( \frac{1}{\sqrt{k}} \norm{T_k}_F \right)\right]^2.
    \end{equation*}
    Then, by taking the expectation and exploiting both the concavity of $\psi$ and~\eqref{eq:expectation_norm_omega_k:F} of Lemma~\ref{lem:6}, we finally obtain the result.
\end{proof}

By exploiting key projection perturbation results from~\cite{DrineasIpsen_LowRankMatrixApproximations_2019}, we are now able to extend our analysis to the general setting,  where $Z$ is drawn from a Gaussian distribution of mean $\widehat{Z} \in \Rnp$ and of covariance matrix $\Covmat(Z)$. This is the second key result of our stochastic analysis proposed next.

\begin{proposition} \label{prop:2}
    Let $A \in \Rnm$ such that $n \geq m$ and $Z \in \Rnp$ a Gaussian matrix of mean $\widehat{Z} \in \Rnp$ and covariance matrix $\Covmat(Z)$, that is, $Z \sim \mathcal{N}(\widehat{Z}, \Covmat(Z))$, satisfying $2 < p \leq \min\set{\rank(A), \, \rank(\Covmat(Z))}$. Let $\pi(Z)$ and $\pi(Z- \widehat{Z})$ denote the orthogonal projections onto the vector spaces spanned by the columns of $Z$ and $Z- \widehat{Z}$, respectively.

    Then, for any $k \in \set{1, \dots, p-2}$, one has
    \begin{align*}
         \expect{\norm{[\I_n - \pi(Z)] A}_{2,F} - \norm{[\I_n - \pi(Z)] \ubar{A}_k}_{2,F}} \leq \frac{e \sqrt{r}}{r-p} \frac{\norm{\widehat{Z}}_2}{\sqrt{\lambda_r}} \norm{A_k}_{2,F} + \expect{\norm{[\I_n - \pi(Z- \widehat{Z})] A_k}_{2,F}},
    \end{align*}
    where $r$ denotes the rank of $\Covmat(Z)$ and $\lambda_r$ the smallest nonzero eigenvalue of $\Covmat(Z)$.
\end{proposition}
\begin{proof}
    First, we begin with the trivial observation that
    \begin{equation*}
        \norm{[\I_n - \pi(Z)] A}_{2, F} - \norm{[\I_n - \pi(Z)] \ubar{A}_k}_{2, F} \leq \norm{[\I_n - \pi(Z)] A_k}_{2,F}.
    \end{equation*}
    Let us define $W = Z - \widehat{Z}$, which is a centered Gaussian matrix with covariance matrix $\Covmat(W) = \Covmat(Z)$. It yields
    \begin{align*}
        \norm{[\I_n - \pi(Z)] A_k}_{2,F} & = \norm{[\I_n - \pi(W) + \pi(W) - \pi(Z)] A_k}_{2,F}, \\
        & = \norm{[\I_n - \pi(W)] A_k + [\pi(W) - \pi(Z)] A_k}_{2,F}, \\
        & \leq  \norm{[\I_n - \pi(W)] A_k}_{2,F} + \norm{[\pi(W) - \pi(Z)] A_k}_{2,F}, \\
        & \leq  \norm{[\I_n - \pi(W)] A_k}_{2,F} + \norm{[\pi(W) - \pi(Z)]}_2 \norm{A_k}_{2,F}, \\
        & = \norm{[\I_n - \pi(W)] A_k}_{2,F} + \norm{[\I_n - \pi(Z)] \pi(W)}_2 \norm{A_k}_{2,F},
    \end{align*}
    where the last inequality follows from~\cite[Lemma 1.5]{DrineasIpsen_LowRankMatrixApproximations_2019}. Then, adapting the proof of~\cite[Theorem 2.2]{DrineasIpsen_LowRankMatrixApproximations_2019}, we observe that
    \begin{align*}
        [\I_n - \pi(Z)] \pi(W) & = [\I_n - \pi(Z)] W W\pinv, \\
        & = [\I_n - \pi(Z)] (W - Z + Z) W\pinv, \\
        & = [\I_n - \pi(Z)] (W - Z) W\pinv  + [\I_n - \pi(Z)] Z W\pinv, \\
        & = [\I_n - \pi(Z)] (W - Z) W\pinv.
    \end{align*}
    Therefore, taking the spectral norm and using the submultiplicativity yields
    \begin{equation*}
        \norm{[\I_n - \pi(Z)] \pi(W)}_2 \leq
        \norm{W - Z}_2 \norm{W\pinv}_2 = \norm{\widehat{Z}}_2 \norm{W\pinv}_2.
    \end{equation*}
    Overall, we have
    \begin{equation*}
        \norm{[\I_n - \pi(Z)] A_k}_{2,F} \leq
        \norm{[\I_n - \pi(W)] A_k}_{2,F} + \norm{\widehat{Z}}_2 \norm{W\pinv}_2 \norm{A_k}_{2,F}.
    \end{equation*}
    Since $W \in \Rnp$ is full column rank, one has $W\pinv = (W\tra W)\inv W\tra$ and due to the unitarily invariance of the spectral norm,
    \begin{equation*}
        \norm{W\pinv}_2 = \norm{(W\tra W)\halfinv}_2 = \sqrt{\norm{(W\tra W)\inv}_2}.
    \end{equation*}
    Given $r = \rank(\Covmat{Z})$, the covariance matrix $\Covmat(Z)$ has the following thin eigendecomposition, $\Covmat(Z) = Q_r \Lambda_r Q_r\tra $
    where $Q_r \in \Rnr$ is an orthogonal matrix containing the first $\lambda_r$ eigenvectors of $\Covmat(Z)$ and $\Lambda_r = \diag(\lambda_1, \dots, \lambda_r)$ contains the corresponding nonzero eigenvalues sorted in decreasing order, i.e. $\lambda_1 \geq \dots \geq \lambda_r >0$. Then, using~\eqref{eq:gaussian_matrix_relation}, one has $W = \Covmat(Z)\half G$ with $G \in \Rnp$ such that $G \sim \mathcal{N}(0, \I_n)$. Thus,
    \begin{equation*}
        W\tra W = G\tra \Covmat(Z) G = G\tra Q_r \Lambda_r Q_r\tra G = Y_r\tra \Lambda_r Y_r \ggeq \lambda_r Y_r\tra Y_r,
    \end{equation*}
    where $Y_r = Q_r\tra G \sim \mathcal{N}(0, \I_r)$.
    Then, $(W\tra W)\inv \lleq \lambda_r\inv (Y_r\tra Y_r)\inv$ leading to
    \begin{equation*}
        \norm{W\pinv}_2 \leq
        \frac{1}{\sqrt{\lambda_r}} \sqrt{\norm{(Y_r\tra Y_r)\inv}_2} =
        \frac{1}{\sqrt{\lambda_r}} \norm{Y_r\pinv}_2 =
        \frac{1}{\sqrt{\lambda_r}} \norm{(Y_r\tra)\pinv}_2,
    \end{equation*}
    with $Y_r\tra \in \R^{p \times r}$ (with $p \leq r$ by assumption) and $Y_r\tra \sim \mathcal{N}(0, \I_p)$. We then take the expectation and apply Lemma~\ref{lem:5} (with $N = \I_r$) to deduce an upper bound of $\expect{\norm{(Y_r\tra)\pinv}_2}$, which completes the proof.
\end{proof}

We are now able to state the main theorems concerning the error bounds in expectation. Namely, given $Z \sim \mathcal{N}(\widehat{Z}, \Covmat(Z))$, we target to bound the following quantity
\begin{equation*}
    \expect{\norm{[\I_n - \pi(Z)] A}_{2,F} - \norm{[\I_n - \pi(Z)] \ubar{A}_k}_{2,F}}.
\end{equation*}


The proof of the next two theorems can be  straightforwardly obtained by using the inequality $\sqrt{a^2 + b^2} \leq a + b$ for any real positive scalars $a$ and $b$, and then by taking the expectation of~\eqref{deterministic_bound_F_2} of Theorem~\ref{thm:1}. The right-hand side terms are bounded using Lemma~\ref{lem:6} (with $N = \Sigma_k$) and Propositions~\ref{prop:1} and~\ref{prop:2}, respectively. We first state our error bound in the Frobenius norm case.

\begin{theorem}\label{thm:3}(Average analysis error bound in Frobenius norm)
    Let $A \in \Rnm$ such that $n \geq m$ and $Z \in \Rnp$ a Gaussian matrix of mean $\widehat{Z} \in \Rnp$ and covariance matrix $\Covmat(Z)$, that is, $Z \sim \mathcal{N}(\widehat{Z}, \Covmat(Z))$, satisfying $2 < p \leq \min\set{\rank(A), \, \rank(\Covmat(Z))}$. Let $\pi(Z)$ denote the orthogonal projection onto the vector space spanned by the columns of $Z$. For a given $k \in \set{1, \dots, p-2}$, set $\Omega_k = U_k\tra (Z - \widehat{Z}) \in \R^{k \times p}$, $\ubar{\Omega}_k = \ubar{U}_k\tra  (Z - \widehat{Z}) \in \R^{(n-k) \times p}$ such that $\Omega_k$ is full row rank (i.e. $\rank(\Omega_k) = k$). Let  $\varphi : x \mapsto x / \sqrt{1 + x^2}$ for $x \geq 0$.

    If the covariance matrix $\Covmat_k(Z)$ is non-singular, then  one has
    \begin{equation} \label{proba_bound_norm_F_exp_b}
        \begin{aligned}
            & \expect{\norm{[\I_n - \pi(Z)] A}_{F} - \norm{[\I_n - \pi(Z)] \ubar{A}_k}_{F} } \leq \\
            & \quad \quad \quad \frac{e \sqrt{r}}{r-p} \frac{\norm{\widehat{Z}}_2}{\sqrt{\lambda_r}} \norm{A_k}_{F} +
            \min \set{\sqrt{a_k},~\sqrt{k}~\varphi \left(\dfrac{1}{\sqrt{k}}~\sqrt{b_k} \right) \norm{\Sigma_k}_2},
        \end{aligned}
    \end{equation}
    where  $r$ denotes the rank of $\Covmat(Z)$, $\lambda_r$ the smallest nonzero eigenvalue of $\Covmat(Z)$ and
    \begin{align*}
        a_k & = \norm{\Covmat_{\perp,k}(Z) [\Covmat_k(Z)]\inv \Sigma_k}_F^2 +
        \frac{\norm{\Covmat\left(\ubar{\Omega}_k \mid \Omega_k\right)\half}_F^2\norm{(\Sigma_k\tra [\Covmat_k(Z)]\inv \Sigma_k)\half}_F^2}{p-k-1}, \\
        b_k & = \norm{\Covmat_{\perp,k}(Z) [\Covmat_k(Z)]\inv}^2_{F} + \frac{\norm{\Covmat\left(\ubar{\Omega}_k \mid \Omega_k\right)\half}_F^2\norm{[\Covmat_k(Z)]\halfinv}_F^2}{p-k-1},
    \end{align*}
    with {$\Covmat_k(Z) = U_k\tra \Covmat(Z) U_k$, $\Covmat_{\perp,k}(Z)= \ubar{U}_k\tra \Covmat(Z) U_k$,   $\uwidebar{\Covmat}_k(Z) = \ubar{U}_k\tra \Covmat(Z) \ubar{U}_k$ and $\Covmat\left(\ubar{\Omega}_k \mid \Omega_k\right) = \uwidebar{\Covmat}_k(Z) - \Covmat_{\perp,k}(Z) \left[\Covmat_k(Z)\right]\inv \Covmat_{\perp,k}(Z)\tra$}.
\end{theorem}

\begin{remark}
We note that in the case of $\widehat{Z}=0$, we are able to improve the tightness of our error bound in the Frobenius norm. More precisely, we have
    \begin{align*}
    & \expect{\norm{[\I_n - \pi(Z)] A}^2_{F} - \norm{[\I_n - \pi(Z)] \ubar{A}_k}^2_{F} } \leq \min \set{a_k,~k ~\varphi \left( \dfrac{1}{\sqrt{k}} ~\sqrt{b_k} \right)^2 \norm{\Sigma_k}^2_2},
    \end{align*}
    where $a_k$ and $b_k$ are defined in Theorem~\ref{thm:3}.
\end{remark}

Similarly, we state our error bound in the spectral norm next.

\begin{theorem}\label{thm:4}(Average analysis error bound in spectral norm)
    Let $A \in \Rnm$ such that $n \geq m$ and $Z \in \Rnp$ a Gaussian matrix of mean $\widehat{Z} \in \Rnp$ and covariance matrix $\Covmat(Z)$, that is, $Z \sim \mathcal{N}(\widehat{Z}, \Covmat(Z))$, satisfying $2 < p \leq \min\set{\rank(A), \, \rank(\Covmat(Z))}$. Let $\pi(Z)$ denote the orthogonal projection onto the vector space spanned by the columns of $Z$. For a given $k \in \set{1, \dots, p-2}$, set $\Omega_k = U_k\tra (Z - \widehat{Z}) \in \R^{k \times p}$, $\ubar{\Omega}_k = \ubar{U}_k\tra (Z -\widehat{Z}) \in \R^{(n-k) \times p}$ such that $\Omega_k$ is full row rank (i.e. $\rank(\Omega_k) = k$). Let $\varphi : x \mapsto x / \sqrt{1 + x^2}$ for $x \geq 0$.

    If the covariance matrix $\Covmat_k(Z)$ is non-singular, then  one has
    \begin{equation} \label{proba_bound_norm_2_exp_b}
        \begin{aligned}
            & \expect{\norm{[\I_n - \pi(Z)] A}_{2} - \norm{[\I_n - \pi(Z)] \ubar{A}_k}_{2} } \leq \frac{e \sqrt{r}}{r-p} \frac{\norm{\widehat{Z}}_2}{\sqrt{\lambda_r}} \norm{A_k}_{2} +
            \min \set{c_k, ~\varphi(d_k) \norm{\Sigma_k}_2},
        \end{aligned}
    \end{equation}
    where  $r$ denotes the rank of $\Covmat(Z)$, $\lambda_r$ the smallest nonzero eigenvalue of $\Covmat(Z)$ and
    \begin{align*}
        c_k & = \norm{\Covmat_{\perp,k}(Z) [\Covmat_k(Z)]\inv \Sigma_k}_{2} +
        \frac{\norm{\Covmat\left(\ubar{\Omega}_k \mid \Omega_k\right)\half}_2 \norm{(\Sigma_k\tra [\Covmat_k(Z)]\inv \Sigma_k)\half}_F}{\sqrt{p-k-1}} \\
        & \quad + \frac{e\sqrt{p}}{{p-k}}
        \norm{\Covmat\left(\ubar{\Omega}_k \mid \Omega_k\right)\half}_F \norm{(\Sigma_k\tra [\Covmat_k(Z)]\inv \Sigma_k)\half}_2 ,\\
        d_k & = \norm{\Covmat_{\perp,k}(Z) [\Covmat_k(Z)]\inv}_{2} + \frac{\norm{\Covmat\left(\ubar{\Omega}_k \mid \Omega_k\right)\half}_2 \norm{[\Covmat_k(Z)]\halfinv}_F}{\sqrt{p-k-1}}\\
        & \quad + \frac{e\sqrt{p}}{{p-k}}
        \norm{\Covmat\left(\ubar{\Omega}_k \mid \Omega_k\right)\half}_F \norm{[\Covmat_k(Z)]\halfinv}_2,
    \end{align*}
    with {$\Covmat_k(Z) = U_k\tra \Covmat(Z) U_k$,  $\Covmat_{\perp,k}(Z) = \ubar{U}_k\tra \Covmat(Z) U_k$,  $\uwidebar{\Covmat}_k(Z) = \ubar{U}_k\tra \Covmat(Z) \ubar{U}_k$ and $\Covmat\left(\ubar{\Omega}_k \mid \Omega_k\right) = \uwidebar{\Covmat}_k(Z) - \Covmat_{\perp,k}(Z) \left[\Covmat_k(Z)\right]\inv \Covmat_{\perp,k}(Z)\tra$}.
\end{theorem}

As in Section~\ref{sec:2:sub:1}, if we target to bound the quantity of interest $\expect{\norm{[\I_n - \pi(Z)] A}_2} - \norm{\ubar{A}_k}_2$ instead of $\expect{\norm{[\I_n - \pi(Z)] A}_2 - \norm{[\I_n - \pi(Z)]\ubar{A}_k}_2}$, it is then possible to significantly improve the error bound given in Theorem~\ref{thm:4}. The proof is similar to the proof of Theorem~\ref{thm:4}, exception made that we take the expectation of the result given in Theorem~\ref{thm:2} and that the right-hand side terms are now bounded using Lemma~\ref{lem:6} with the choice $N = \widehat{\Sigma}_k$. We give this result in the next theorem.

\begin{theorem} \label{thm:5}(Average analysis error bound in spectral norm, improved bound)
    Let $A \in \Rnm$ such that $n \geq m$ and $Z \in \Rnp$ a Gaussian matrix of mean $\widehat{Z} \in \Rnp$ and covariance matrix $\Covmat(Z)$, that is, $Z \sim \mathcal{N}(\widehat{Z}, \Covmat(Z))$, satisfying $2 < p \leq \min\set{\rank(A), \, \rank(\Covmat(Z))}$. Let $\pi(Z)$ denote the orthogonal projection onto the vector space spanned by the columns of $Z$. For a given $k \in \set{1, \dots, p-2}$, set $\Omega_k = U_k\tra (Z - \widehat{Z}) \in \R^{k \times p}$, $\ubar{\Omega}_k = \ubar{U}_k\tra  (Z - \widehat{Z}) \in \R^{(n-k) \times p}$ such that $\Omega_k$ is full row rank (i.e. $\rank(\Omega_k) = k$). Let  $\varphi : x \mapsto x / \sqrt{1 + x^2}$ for $x \geq 0$.

    If the covariance matrix $\Covmat_k(Z)$ is non-singular, then, one has
    \begin{equation} \label{proba_bound_norm_2_exp}
             \expect{\norm{[\I_n - \pi(Z)] A}_{2}} - \norm{\ubar{A}_k}_{2}  \leq \frac{e \sqrt{r}}{r-p} \frac{\norm{\widehat{Z}}_2}{\sqrt{\lambda_r}} \norm{A_k}_{2}+
            \min \set{\widehat{c}_k,~\varphi(\widehat{d}_k) \norm{\widehat{\Sigma}_k}_2},
    \end{equation}
    where $\widehat{\Sigma}_k = \left(\Sigma_k^2 - \sigma_{k+1}^2 \I_k\right)\half$,  $r$ denotes the rank of $\Covmat(Z)$, $\lambda_r$ the smallest nonzero eigenvalue of $\Covmat(Z)$ and
    \begin{align*}
        \widehat{c}_k & = \norm{\Covmat_{\perp,k}(Z) [\Covmat_k(Z)]\inv \widehat{\Sigma}_k}_{2} +
        \frac{\norm{\Covmat\left(\ubar{\Omega}_k \mid \Omega_k\right)\half}_2 \norm{(\widehat\Sigma_k\tra [\Covmat_k(Z)]\inv \widehat\Sigma_k)\half}_F}{\sqrt{p-k-1}} \\
        & \quad + \frac{e\sqrt{p}}{{p-k}} \norm{\Covmat\left(\ubar{\Omega}_k \mid \Omega_k\right)\half}_F \norm{(\widehat\Sigma_k\tra [\Covmat_k(Z)]\inv \widehat\Sigma_k)\half}_2 ,\\
        \widehat{d}_k & = \norm{\Covmat_{\perp,k}(Z) [\Covmat_k(Z)]\inv}_{2} + \frac{\norm{\Covmat\left(\ubar{\Omega}_k \mid \Omega_k\right)\half}_2 \norm{[\Covmat_k(Z)]\halfinv}_F}{\sqrt{p-k-1}}\\
        & \quad + \frac{e\sqrt{p}}{{p-k}} \norm{\Covmat\left(\ubar{\Omega}_k \mid \Omega_k\right)\half}_F \norm{[\Covmat_k(Z)]\halfinv}_2,
    \end{align*}
    with {$\Covmat_k(Z) = U_k\tra \Covmat(Z) U_k$,  $\Covmat_{\perp,k}(Z)= \ubar{U}_k\tra \Covmat(Z) U_k$ and $\uwidebar{\Covmat}_k(Z) = \ubar{U}_k\tra \Covmat(Z) \ubar{U}_k$ and  $\Covmat\left(\ubar{\Omega}_k \mid \Omega_k\right) = \uwidebar{\Covmat}_k(Z) - \Covmat_{\perp,k}(Z) \left[\Covmat_k(Z)\right]\inv \Covmat_{\perp,k}(Z)\tra$}.
\end{theorem}

\begin{remark}
    Since $\norm{\widehat{\Sigma}_k}_2 \leq \norm{\Sigma_k}_2$ and due to the partial ordering property, we deduce that $\widehat{c}_k \le c_k$ and remark that $\widehat{d}_k = d_k$. Hence if one targets to bound $\expect{\norm{[\I_n - \pi(Z)] A}_{2}} - \norm{\ubar{A}_k}_2$, then the error  bound~\eqref{proba_bound_norm_2_exp} is tighter than~\eqref{proba_bound_norm_2_exp_b}. Hence Theorem~\ref{thm:5} provides a new error bound in the spectral norm in a fairly general setting.
\end{remark}

We provide in Section~\ref{sec:4} numerical illustrations to highlight the potential of the error bounds.

\section{Application to the  Randomized Singular Value Decomposition} \label{sec:3}

As an illustration, we consider the Randomized Singular Value Decomposition (RSVD), a popular algorithm for obtaining a low-rank approximation to a given matrix~\cite{HalkoMartinssonEtAl_FindingStructureRandomness_2011,MartinssonTropp_RandomizedNumericalLinear_2020}. In this method, $Z \in \Rnp$ is constructed to approximate the $k$ dominant left singular vectors of $A$. Given $q \in \N$ and $A_q = (A A\tra)^q A$, the Randomized Singular Value Decomposition considers $Z = A_q G$. In the single-pass setting ($q=0$), error bounds have been notably provided in~\cite{HalkoMartinssonEtAl_FindingStructureRandomness_2011,Halko_RandomizedMethodsComputing_2012} for the spectral norm. For $q \in \N^{\star}$, we provide error bounds which, to the best of our knowledge, are new. As detailed next, the relation $Z \sim \mathcal{N}(0, A_q A_q\tra)$ leads to significant simplifications in the expressions of the error bounds given in Theorems~\ref{thm:3},~\ref{thm:4} and~\ref{thm:5}, respectively. Deriving these error bounds implies to first express the projected covariance matrices defined in~\eqref{eq:partitioning_covariance_matrix}. This is the main purpose of the next lemma.

\begin{lemma} \label{lem:3:rsvd}
    Let $A \in \Rnm$ such that $n \geq m$ and $Z \in \Rnp$ with $2 < p \leq \rank(A)$ such that $Z = A_q G$ with $A_q = (A A\tra)^q A \in \Rnm$ ($q \in \N$) and $G \in \Rmp$ a standard Gaussian matrix ($G \sim \mathcal{N}(0, \I_m)$). For a given integer $k \in \set{1, \dots, p}$, set $\Omega_k = U_k\tra Z$ and $\ubar{\Omega}_k = \ubar{U}_k\tra Z$. The projected covariance matrices are then given by
    \begin{equation}\label{eq:partitioning_covariance_matrix:rsvd}
        \begin{aligned}
            \Covmat_k(Z) & = \Sigma_k^{4q+2}, \\
            \Covmat_{\perp,k}(Z) & = 0,   \\
            \uwidebar{\Covmat}_k(Z) & = (\ubar{\Sigma}_k \ubar{\Sigma}_k\tra)^{2q+1}.
        \end{aligned}
    \end{equation}
    Then the random matrix $\ubar{\Omega}_k$ conditioned by $\Omega_k$ follows a Gaussian distribution of mean
    \begin{equation} \label{eq:conditional_dsitribution:mean:rsvd}
        \expect{\ubar{\Omega}_k \mid \Omega_k} = 0,
    \end{equation}
    and of covariance matrix
    \begin{equation} \label{eq:conditional_dsitribution:cov:rsvd}
        \Covmat\left(\ubar{\Omega}_k \mid \Omega_k\right) = (\ubar{\Sigma}_k \ubar{\Sigma}_k\tra)^{2q+1}.
    \end{equation}
\end{lemma}
\begin{proof}
    A straightforward calculation gives
    $\Omega_k = \Sigma_k^{2q+1} V_k\tra G$. Hence $\Omega_k$ has full row rank with probability one~\cite{FengZhang_RankRandomMatrix_2007}. Since $Z \sim \mathcal{N}(0,A_q A_q\tra)$,
    $\Covmat(Z)= A_q A_q\tra$ can be expressed as
    \begin{equation*}
        \Covmat(Z) = U (\Sigma \Sigma\tra)^{2q+1} U^T.
    \end{equation*}
    Thus $\Covmat(Z)$ and $A$ have the same rank and we deduce
    \begin{align*}
        \Covmat_k(Z) & = \Sigma_k^{4q+2}, \\
        \Covmat_{\perp,k}(Z) & = 0, \\
        \uwidebar{\Covmat}_k(Z) & = (\ubar{\Sigma}_k \ubar{\Sigma}_k\tra)^{2q+1}.
    \end{align*}
    We apply relations (\ref{eq:conditional_dsitribution:mean}) and (\ref{eq:conditional_dsitribution:cov}) of Lemma~\ref{lem:3} to deduce
    \begin{align*}
        \expect{ \ubar{\Omega}_k \mid \Omega_k} & = 0, \\
        \Covmat\left(\ubar{\Omega}_k \mid \Omega_k\right) & = (\ubar{\Sigma}_k \ubar{\Sigma}_k\tra)^{2q+1},
    \end{align*}
    which concludes the proof.
\end{proof}

\begin{remark}
    From Lemma~\ref{lem:3:rsvd}, we deduce that the projected covariance matrix $\Covmat_k(Z)$ is non-singular, allowing us to apply Theorems~\ref{thm:3},~\ref{thm:4} and~\ref{thm:5}, respectively. Since $\Covmat_{\perp,k}(Z) = 0$, we also note that $\ubar{\Omega}_k$ and $\Omega_k$ become statistically independent.
\end{remark}

With the help of relations~\eqref{eq:partitioning_covariance_matrix:rsvd},~\eqref{eq:conditional_dsitribution:mean:rsvd} and~\eqref{eq:conditional_dsitribution:cov:rsvd},  we specialize the main theorems proposed in Section~\ref{sec:2:sub:2} to the setting of the Randomized Singular Value Decomposition. To enhance the readability of the constants arising in the error bounds, we introduce the singular value ratios related to the matrix $A$ as
\begin{equation} \label{eq:ratio_singular_values}
    \gamma_i = \frac{\sigma_{k+1}}{\sigma_i}, \quad i = 1, \dots, \rank{(A)}.
\end{equation}

\subsection{Error bound in expectation in Frobenius norm} \label{sec:rsvd:frob}


We first consider the case of the Frobenius norm. A direct application of Theorem~\ref{thm:3} and Lemma~\ref{lem:3:rsvd} then leads to the following corollary.

\begin{corollary}(Average analysis error bound in Frobenius norm, RSVD) \label{thm:3:rsvd}
    Let $A \in \Rnm$ such that $n \geq m$ and $Z \in \Rnp$ with $2 < p \leq \rank(A)$ such that $Z = A_q G$ with $A_q = (A A\tra)^q A \in \Rnm$ ($q \in \N$) and $G \in \Rmp$ a standard Gaussian matrix ($G \sim \mathcal{N}(0, \I_m)$). Let $\pi(Z)$ denote the orthogonal projection onto the vector space spanned by the columns of $Z$. For a given $k \in \set{1, \dots, p-2}$, set $\Omega_k = U_k\tra Z \in \R^{k \times p}$ such that $\Omega_k$ is full row rank (i.e. $\rank(\Omega_k) = k$). Let  $\varphi : x \mapsto x / \sqrt{1 + x^2}$ for $x \geq 0$.

    Then, we have
    \begin{equation} \label{thm:3:rsvd:bound}
        \expect{\norm{[\I_n - \pi(Z)] A}_{F} - \norm{[\I_n - \pi(Z)] \ubar{A}_k}_{F} } \leq
        \min \set{\sqrt{a_k}, ~\sqrt{k} ~\varphi \left( \dfrac{1}{\sqrt{k}} ~\sqrt{b_k} \right) \sigma_1},
    \end{equation}
    where
    \begin{align*}
        a_k & = \frac{\sigma_{k+1}^2}{p-k-1} \left( \sum_{i=k+1}^{\rank{(A)}} \frac{1}{\gamma_i^{4q+2}} \right) \left(\sum_{i=1}^k {\gamma_i^{4q}} \right), \\
        b_k & = \frac{1}{p-k-1} \left( \sum_{i=k+1}^{\rank{(A)}} \displaystyle \frac{1}{\gamma_i^{4q+2}}
        \right) \left( \sum_{i=1}^k \gamma_i^{4q+2} \right),
    \end{align*}
with the singular value ratios $\gamma_i$ given in~\eqref{eq:ratio_singular_values}.
\end{corollary}


\begin{remark}
    In~\cite{BoulleTownsend_LearningEllipticPartial_2022}, Boullé and Townsend have considered the low-rank approximation of partial differential operators in infinite dimension using a zero mean and a general covariance matrix for the random matrix variable. They have recently improved their error bound to the finite-dimensional case using the Frobenius norm in~\cite{BoulleTownsend_GeneralizationRandomizedSingular_2022}. It can be shown that the error bound~\eqref{thm:3:rsvd:bound} provided in Corollary~\ref{thm:3:rsvd} is always tighter. An illustration is given in Section~\ref{sec:numerical:expect_comparison}.
\end{remark}

\subsection{Error bounds in expectation in spectral norm} \label{sec:rsvd:spec}

We next consider the case of the spectral norm. A straightforward application of Theorem~\ref{thm:4} and Lemma~\ref{lem:3:rsvd} first leads to the following corollary.

\begin{corollary}(Average analysis error bound in spectral norm, RSVD) \label{thm:4:rsvd}
    Let $A \in \Rnm$ such that $n \geq m$ and $Z \in \Rnp$  with $2 < p \leq \rank(A)$ such that $Z = A_q G$ with $A_q = (A A\tra)^q A \in \Rnm$ ($q \in \N $) and $G \in \Rmp$ a standard Gaussian matrix ($G \sim \mathcal{N}(0, \I_m)$). Let $\pi(Z)$ denote the orthogonal projection onto the vector space spanned by the columns of $Z$. For a given $k \in \set{1, \dots, p-2}$, set $\Omega_k = U_k\tra Z \in \R^{k \times p}$, such that $\Omega_k$ is full row rank (i.e. $\rank(\Omega_k) = k$). Let  $\varphi : x \mapsto x / \sqrt{1 + x^2}$ for $x \geq 0$.

    Then, we have
    \begin{equation} \label{thm:4:rsvd:bound}
        \expect{\norm{[\I_n - \pi(Z)] A}_{2} - \norm{[\I_n - \pi(Z)] \ubar{A}_k}_{2}}
        \leq \min \set{c_k, ~\varphi(d_k) \sigma_1},
    \end{equation}

    where
    \begin{align*}
        c_k & = \frac{\sigma_{k+1}}{\sqrt{p-k-1}} \left(\sum_{i=1}^k {\gamma_i^{4q}}\right)\half +
        \sigma_k \left(\sum_{i=k+1}^{\rank{(A)}} \left(\frac{\sigma_i}{\sigma_k}\right)^{4q+2}\right)\half \frac{e \sqrt{p}}{p-k},\\
        d_k & = \frac{1}{\sqrt{p-k-1}} \left(\sum_{i=1}^k {\gamma_i^{4q+2}}\right)\half +
        \left(\sum_{i=k+1}^{\rank{(A)}} \left(\frac{\sigma_i}{\sigma_k}\right)^{4q+2}\right)\half \frac{e \sqrt{p}}{p-k},
    \end{align*}
with the singular value ratios $\gamma_i$ given in~\eqref{eq:ratio_singular_values}.
\end{corollary}

We finally state the improved error bound related to Theorem~\ref{thm:5} and Lemma~\ref{lem:3:rsvd} in the next corollary.

\begin{corollary}(Average analysis error bound in spectral norm, improved bound, RSVD) \label{thm:5:rsvd}
    Let $A \in \Rnm$ such that $n \geq m$ and $Z \in \Rnp$ with $2 < p \leq \rank(A)$ such that $Z = A_q G$ with $A_q = (A A\tra)^q A \in \Rnm$ ($q \in \N$) and $G \in \Rmp$ a standard Gaussian matrix ($G \sim \mathcal{N}(0, \I_m)$). Let $\pi(Z)$ denote the orthogonal projection onto the vector space spanned by the columns of $Z$. For a given $k \in \set{1, \dots, p-2}$, set $\Omega_k = U_k\tra Z \in \R^{k \times p}$, such that $\Omega_k$ is full row rank (i.e. $\rank(\Omega_k) = k$). Let  $\varphi : x \mapsto x / \sqrt{1 + x^2}$ for $x \geq 0$.

    Then, we have
    \begin{equation} \label{thm:5:rsvd:bound}
        \expect{\norm{[\I_n - \pi(Z)] A}_{2} - \sigma_{k+1}}
        \leq \min \set{\widehat{c}_k, ~\varphi(\widehat{d}_k) \sqrt{\sigma_1^2 - \sigma_{k+1}^2}},
    \end{equation}
    where
    \begin{align*}
        \widehat{c}_k & = \frac{\sigma_{k+1}}{\sqrt{p-k-1}} \left(\sum_{i=1}^k {\gamma_i^{4q}(1-\gamma_i^{2})}\right)\half + \sqrt{1-\gamma_\ell^2}~
        \sigma_\ell \left(\sum_{i=k+1}^{\rank{(A)}} \left(\frac{\sigma_i}{\sigma_\ell}\right)^{4q+2}\right)\half \frac{e \sqrt{p}}{p-k}, \\
        \widehat{d}_k & = \frac{1}{\sqrt{p-k-1}} \left(\sum_{i=1}^k {\gamma_i^{4q+2}}\right)\half +
        \left(\sum_{i=k+1}^{\rank{(A)}} \left(\frac{\sigma_i}{\sigma_k}\right)^{4q+2}\right)\half \frac{e \sqrt{p}}{p-k},
    \end{align*}
    with
    \begin{equation}\label{eq:index:l}
        \ell =
        \begin{dcases}
            ~ 1, \quad (q=0 \quad ~\mathrm{or~~if}~ \sigma_{k+1} \sqrt{1+1/(2q)} \ge \sigma_1 , \quad (q \in \N^{\star})), \\
            ~ \argmax_i \left( \frac{\sqrt{1 - \gamma_i^2}}{\sigma_i^{2q}}, \frac{\sqrt{1 - \gamma_{i+1}^2}}{\sigma_{i+1}^{2q}}\right), ~\mathrm{with}~ \sigma_{k+1} \sqrt{1+1/(2q)} \in
            [\sigma_i, \sigma_{i+1}], \quad (q \in \N^{\star}), \\
            ~ k, \quad ~\mathrm{if}~ \sigma_{k+1} \sqrt{1+1/(2q)} \le \sigma_k , \quad (q \in \N^{\star}),
        \end{dcases}
    \end{equation}
    with the singular value ratios $\gamma_i$ given in~\eqref{eq:ratio_singular_values}.
\end{corollary}
\begin{proof}
    To deduce the expression of $\widehat{c}_k$, we need to obtain $\norm{(\widehat{\Sigma}_k\tra [\Covmat_k(Z)]\inv \widehat{\Sigma}_k)\half}_2$. Since
    \begin{equation*}
        \widehat{\Sigma}_k\tra [\Covmat_k(Z)]\inv \widehat{\Sigma}_k =
        \diag \left( \frac{\sigma_i^2 - \sigma_{k+1}^2}{\sigma_i^{4q+2}} \right), \quad i = 1, \dots, k,
    \end{equation*}
    we introduce the map $\psi : x \mapsto (x - \sigma_{k+1}^2) / x^{2q+1}$ for $x \geq \sigma_{k+1} > 0$ for $q \in \N^{\star}$. A simple calculation shows that the extremum of $\psi$ is reached for $x = \sigma_{k+1}^2 (1+1/(2q))$. Hence we deduce
    \begin{equation*}
        \norm{(\widehat{\Sigma}_k\tra [\Covmat_k(Z)]\inv \widehat{\Sigma}_k)\half}_2 :=
         \frac{\sqrt{1 - \gamma_\ell^2}}{\sigma_\ell^{2q}},
    \end{equation*}
    with $\ell$ defined in~\eqref{eq:index:l}. The other quantities arising in $\widehat{c}_k$ can be obtained straightforwardly.
\end{proof}

To the best of our knowledge, we note that the error bounds (\ref{thm:3:rsvd:bound}),  (\ref{thm:4:rsvd:bound}) and (\ref{thm:5:rsvd:bound}) provided in Corollaries~\ref{thm:3:rsvd},~\ref{thm:4:rsvd} and~\ref{thm:5:rsvd} are new. In Section~\ref{sec:numerical:expect_comparison:RSVD}, we provide numerical illustrations to show the relevance of these bounds in the context of the Randomized Singular Value Decomposition.

\section{Numerical illustrations} \label{sec:4}

In this section, we aim at illustrating the numerical behaviour of the error bounds introduced in Sections~\ref{sec:2} and~\ref{sec:3}, respectively. We first show the tightness of the error bounds compared to empirical experimental errors in Section~\ref{sec:numerical:proba_and_expect}. Then, we propose a detailed comparison with the state-of-the-art error bounds~\cite{HalkoMartinssonEtAl_FindingStructureRandomness_2011} in Sections~\ref{sec:numerical:expect_comparison} and~\ref{sec:numerical:expect_comparison:RSVD}, respectively. In the following, we consider a square matrix $A$ (with $n=m=1000$) obtained as follows. First, we select two orthogonal matrices $U\in \Rnn$ and $V \in \Rnn$. Each matrix is obtained independently by drawing a random standard Gaussian matrix and taking its $QR$ factorization. Then, given
\begin{equation*}
    \Sigma = \diag ( \underbrace{1, \dots, 1}_{10 \mbox{~times}},~2\halfinv,~3\halfinv,~\dots,~(n - 9)\halfinv ) \in \Rnn,
\end{equation*}
we conduct all the numerical experiments with $A = U \Sigma V\tra$. This test case is directly inspired from~\cite{SaibabaHartEtAl_RandomizedAlgorithmsGeneralized_2021,TroppYurtseverEtAl_PracticalSketchingAlgorithms_2017}. We consider two different target ranks $k$  ($k=5$ and $k=15$, respectively) to allow a variety of results and comments. We believe that performing such an analysis is instructional since, in practice, we are not aware of the ideal value for the target rank $k$. The data and the scripts to reproduce all the numerical results and the  resulting figures are publicly available at \url{https://github.com/a-scotto/R-SVD-Analysis}.




\subsection{Our error bounds in expectation versus the empirical error} \label{sec:numerical:proba_and_expect}

We first focus on the tightness of the error bounds in expectation, given in Theorem~\ref{thm:3} for the Frobenius norm and in Theorem~\ref{thm:4} for the spectral norm, respectively. To quantify their tightness, for a given target rank $k$ and for a fixed value of the oversampling parameter $\varrho(p) = p - k$, we compare our bounds with the  empirical mean error over $100$ samples, i.e.,
\begin{equation*}
     \frac{1}{100} \sum_{i=1}^{100} \left( \norm{[\I_n - \pi(A G_i)]A}_{2,F} - \norm{[\I_n - \pi(A G_i)]\ubar{A}_k}_{2,F} \right),
\end{equation*}
where $G_i \in \Rnp$ ($1 \leq i \leq 100$) are $100$ independent standard Gaussian matrices ($G_i \sim \mathcal{N}(0, \I_n)$). We note that, in this setting, one has $\widehat{Z} = 0$ and $\Covmat(Z) = A A\tra$.

\begin{figure}
    \centering
    \begin{subfigure}{\textwidth}
        \centering
        \includegraphics[scale=0.5]{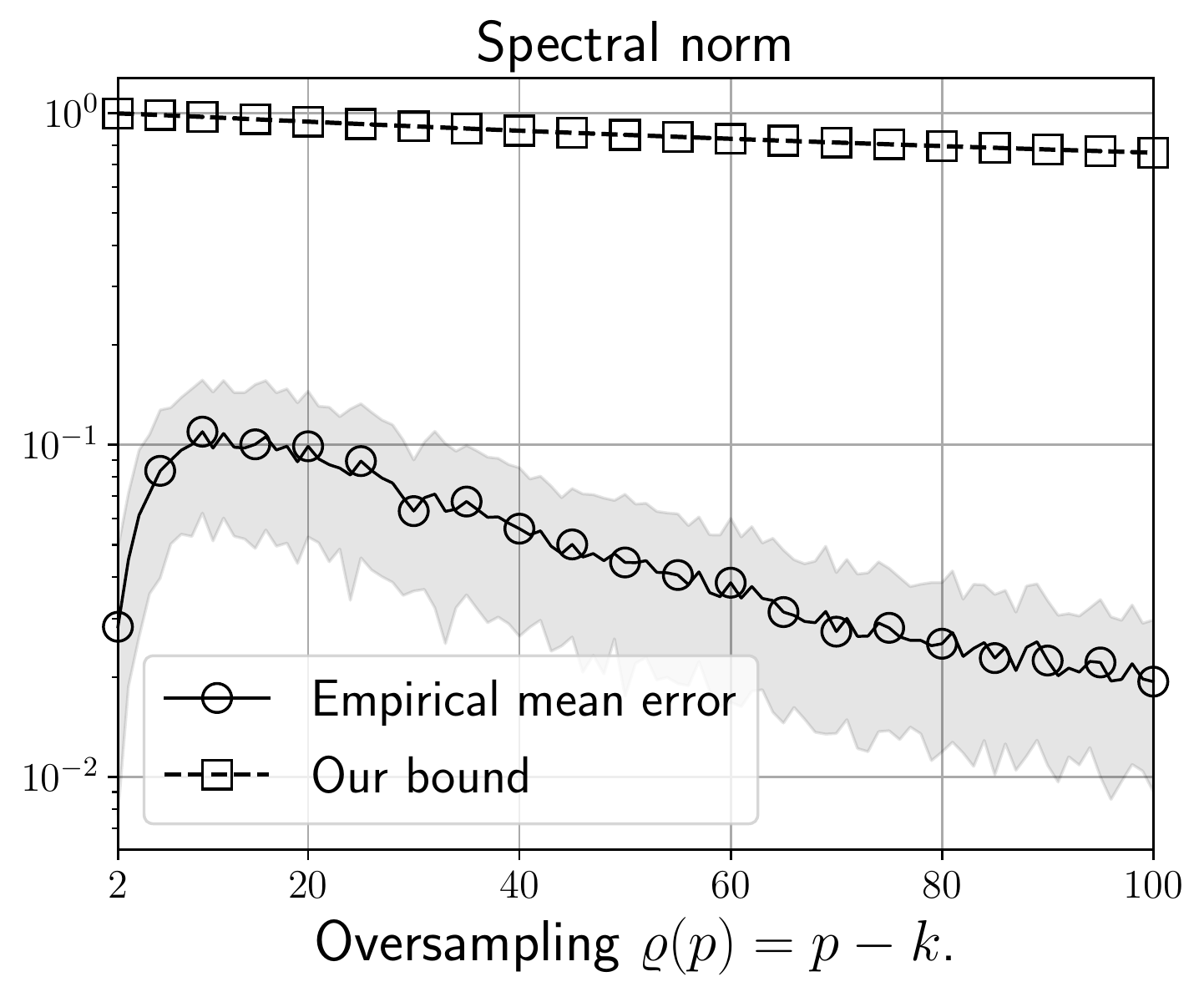} \label{fig:our_bound:spectral:5}
        \includegraphics[scale=0.5]{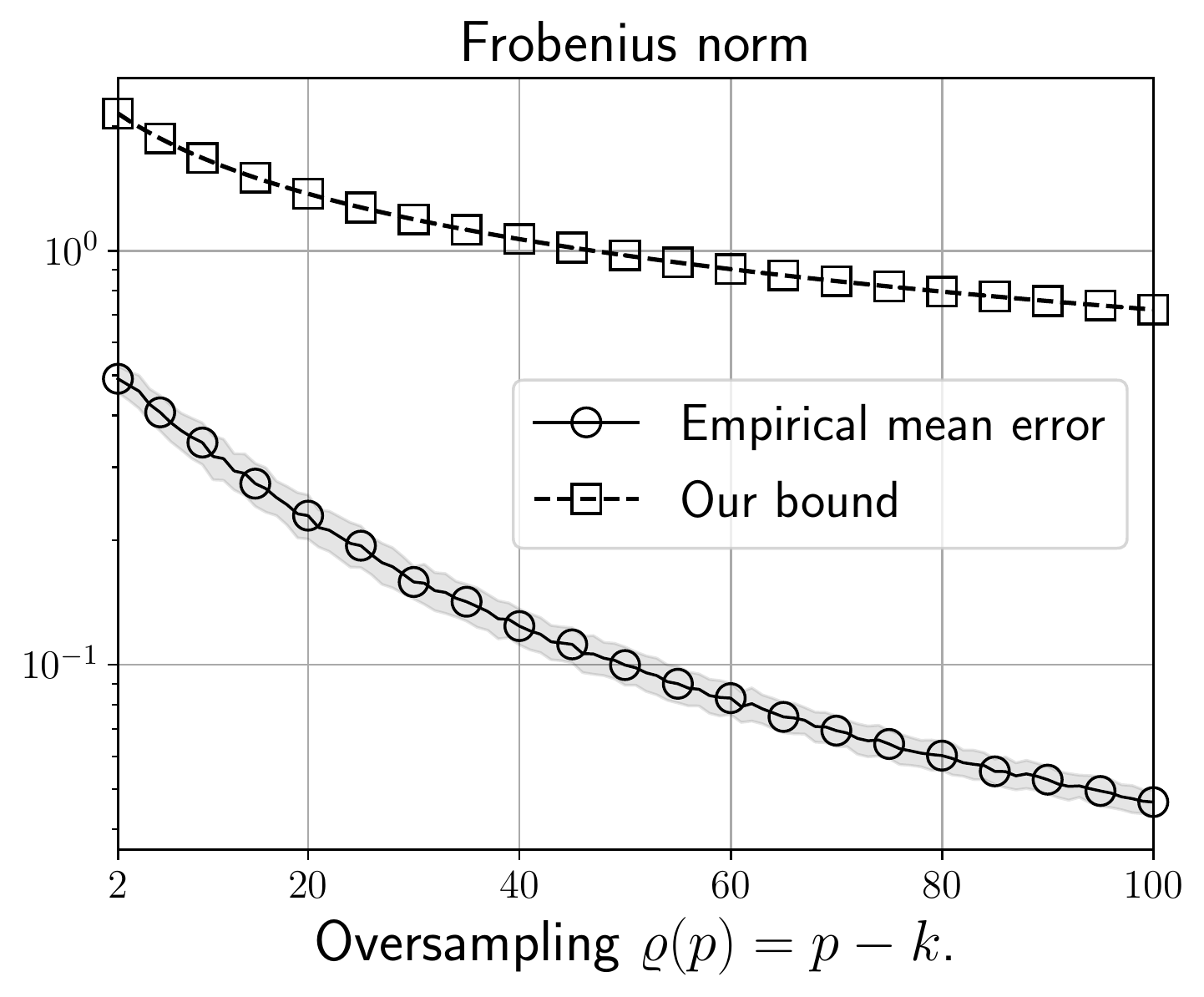}
        \caption{Results for $k=5$.}
        \label{fig:our_bound:5}
    \end{subfigure}
    \begin{subfigure}{\textwidth}
    \centering
        \includegraphics[scale=0.5]{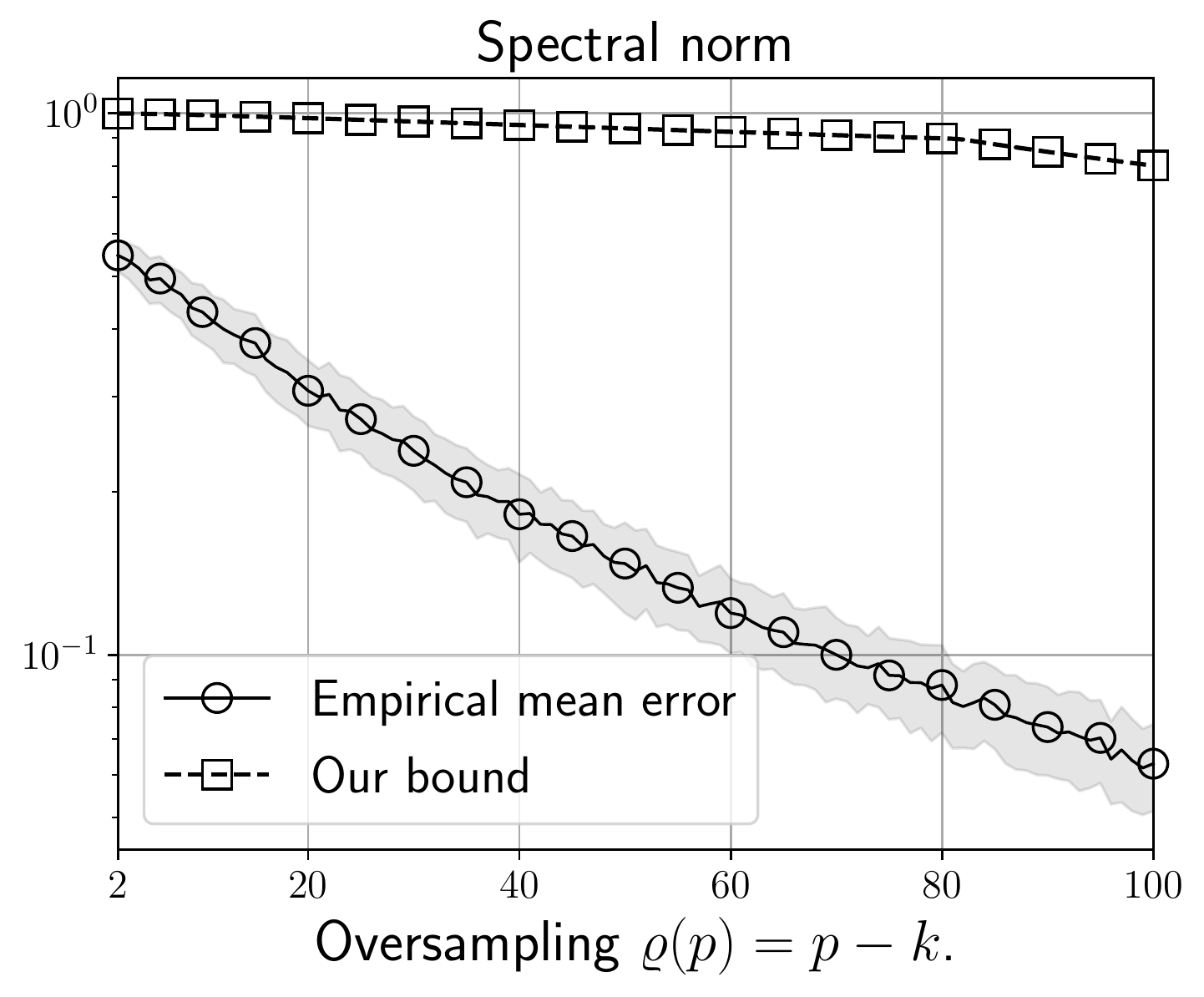}
        \includegraphics[scale=0.5]{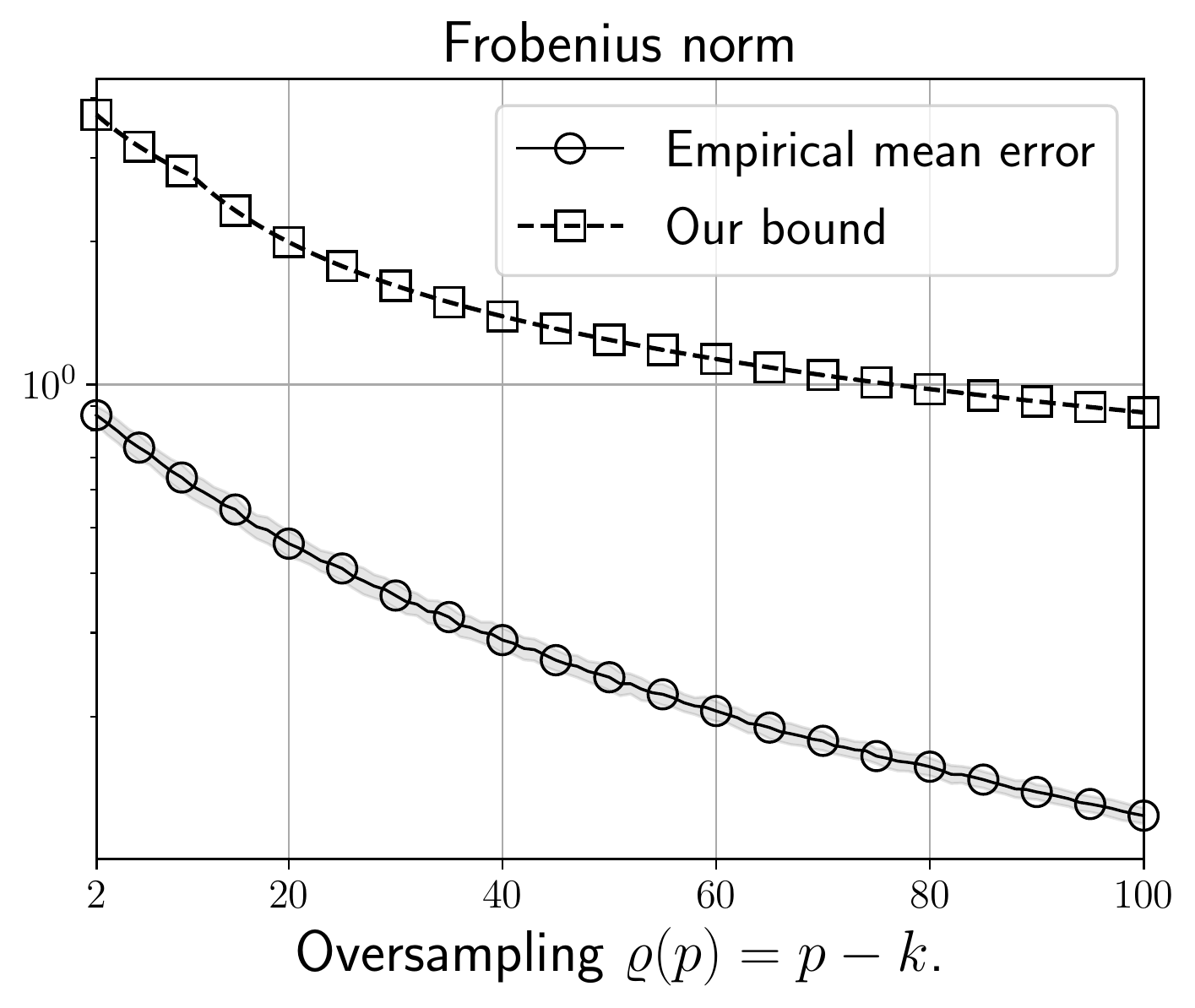}
        \caption{Results for $k=15$.}
        \label{fig:our_bound:15}
    \end{subfigure}
    \caption{Empirical mean error and our error bounds in expectation with respect to the oversampling parameter $\varrho(p) = p - k$. Case of $k=5$ (top), $k=15$ (bottom),  spectral norm (left) and Frobenius norm (right). Statistics on errors are also shown: empirical mean (circle mark) $\pm$ one standard deviation (grey area).}
    \label{fig:expect_proba}
\end{figure}

With respect to the target rank $k$, Figure~\ref{fig:expect_proba} shows the  empirical error as well as our error bounds (in both norms) with respect to the oversampling parameter $\varrho(p)$ for both target ranks $k=5$ and $k=15$, respectively. We remark that, as the number of samples increases, our bounds predict a smaller error in both norms. We also note that the error bounds in expectation in the Frobenius norm are relatively accurate compared to the spectral case. In fact, the decrease rate seems to be slower in the spectral norm compared to the Frobenius norm. This is indeed expected since our error bounds in the spectral norm have been obtained in a looser way compared to the Frobenius norm.

\begin{figure}
    \centering
    \begin{subfigure}{\textwidth}
        \centering
        \includegraphics[scale=0.5]{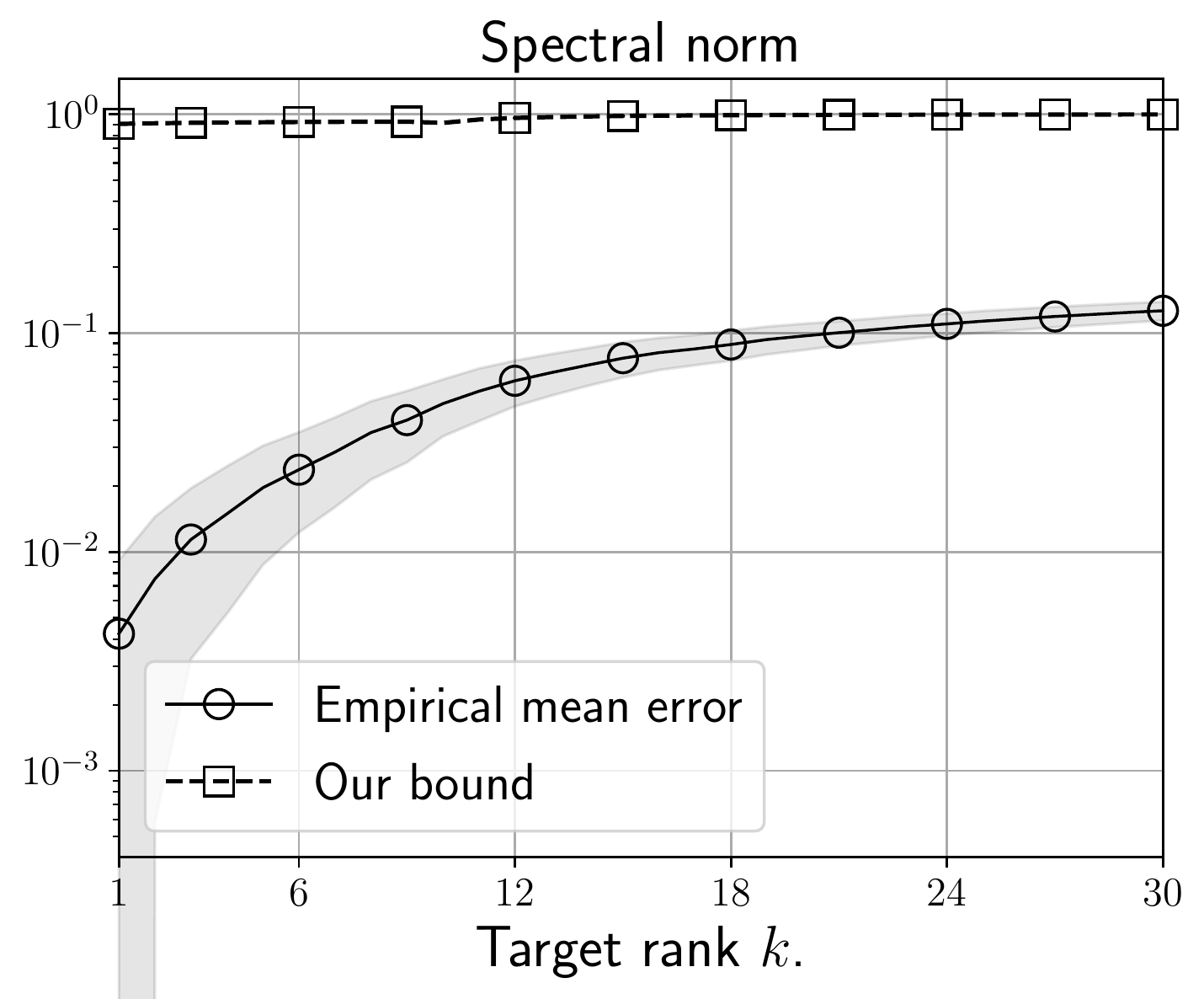}\label{fig:our_bound:spectral:p32}
        \includegraphics[scale=0.5]{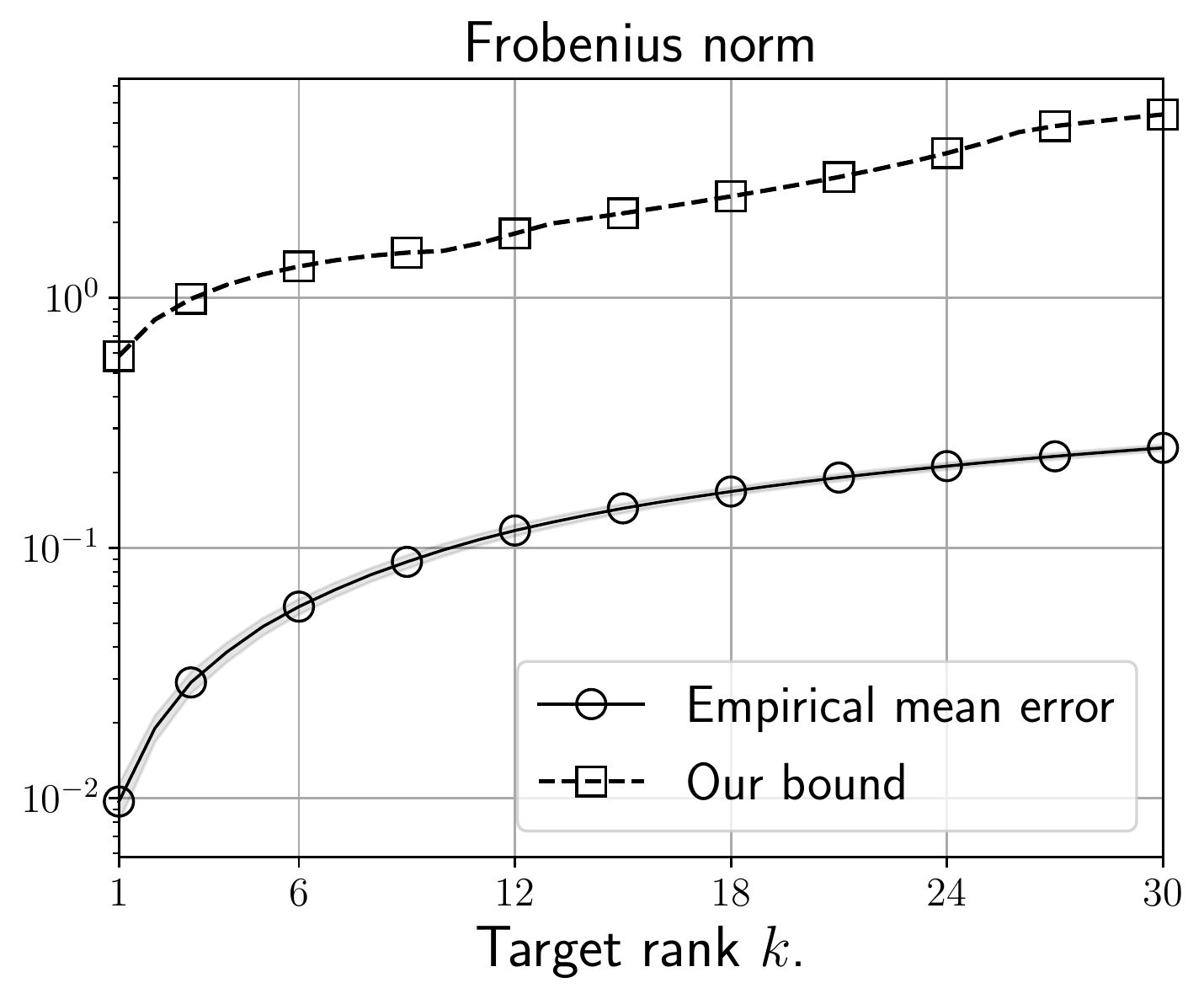}
        \caption{Results for $p=32$.}
        \label{fig:expect_proba_versus:p32}
    \end{subfigure}
    \begin{subfigure}{\textwidth}
        \centering
        \includegraphics[scale=0.5]{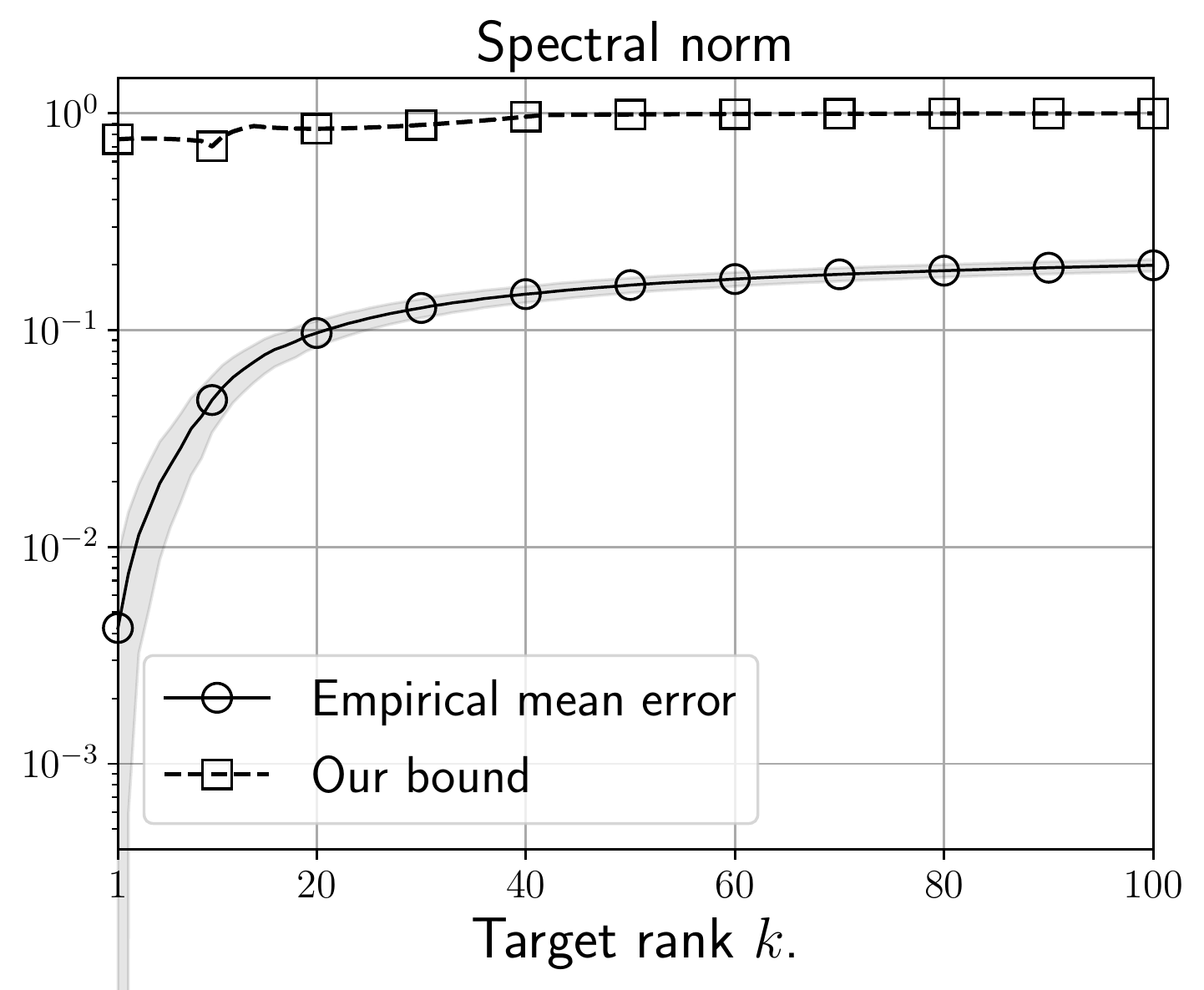}
        \includegraphics[scale=0.5]{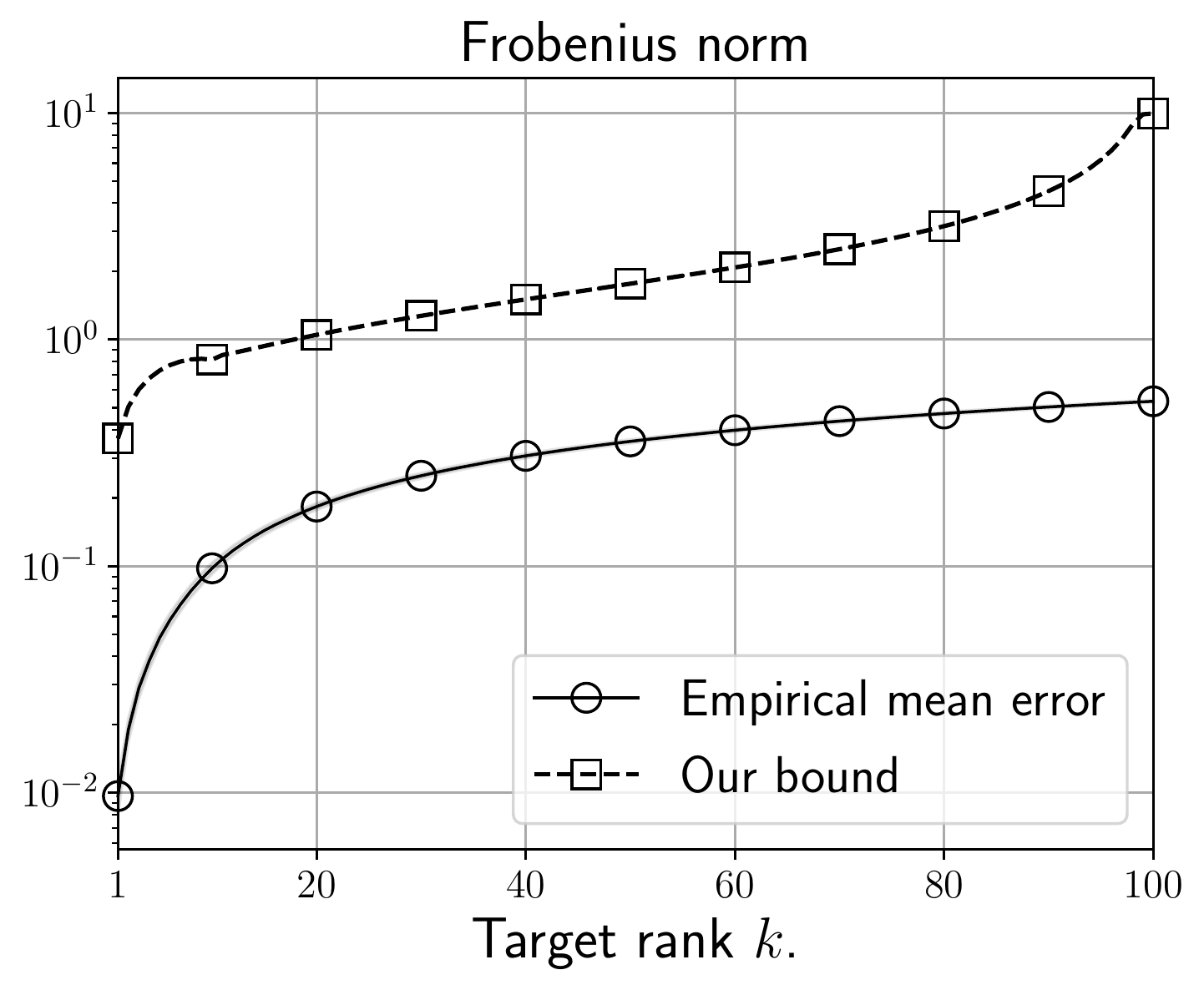}
        \caption{Results for $p=102$.}
        \label{fig:expect_proba_versus:p102}
    \end{subfigure}
    \caption{Empirical mean error and our error bounds in expectation with respect to the target rank $k$. Cases of $p=32$ (top) and of $p=102$ (bottom), spectral norm (left) and Frobenius norm (right). Statistics on errors are also shown: empirical mean (circle mark) $\pm$ one standard deviation (grey area).}
    \label{fig:expect_proba_versus_k}
\end{figure}

With respect to the oversampling parameter $\varrho(p)$, Figure~\ref{fig:expect_proba_versus_k} shows the empirical error and our error bounds with respect to the target rank $k$ for two fixed values of the sample parameter ($p=32$ and $p=102$, respectively). As expected, the error bounds are found to be more accurate for a lower target rank. In the spectral norm, for small target ranks, the gap between  the empirical mean error and our error bound does not seem to be as tight as in the Frobenius norm. We also note that by increasing the value of $p$, our bound reaches a smaller value for the same target rank $k$. For instance, for a target rank of order $20$, our bound is of order $10^0$ for $p=102$, while it exceeds $5\times 10^0$ in the case of $p=32$.


\subsection{Our error bounds in expectation versus the state-of-the-art} \label{sec:numerical:expect_comparison}

We now compare our error bounds with respect to the reference error bounds~\cite[Theorems 10.5 and 10.6]{HalkoMartinssonEtAl_FindingStructureRandomness_2011}. We stress that various reference error bounds exist in the literature but they can be mostly considered as adaptations\footnote{A noticeable exception is the more involved analysis proposed in~\cite{Gu_SubspaceIterationRandomization_2015}, which investigates the quality of the singular value approximation and of the randomized low-rank approximation to a given matrix (see Theorem 5.7 therein, notably). To unify our presentation, we have therefore decided to restrict the comparison to the reference bounds proposed in~\cite{HalkoMartinssonEtAl_FindingStructureRandomness_2011}.} of the error bounds proposed by Halko, Martinsson and Tropp~\cite{HalkoMartinssonEtAl_FindingStructureRandomness_2011} in different settings, see e.g.,~\cite{BoulleTownsend_GeneralizationRandomizedSingular_2022}. 
For simplicity reasons, we later refer to the Halko, Martinsson and Tropp error bounds~\cite{HalkoMartinssonEtAl_FindingStructureRandomness_2011} by  ``\texttt{HMT~bound}". We note that the ``\texttt{HMT bound}" corresponds to an upper bound for the following expected error quantity
\begin{equation} \label{eq:old_metric}
    \expect{\norm{[\I_n - \pi(Z)]A}_{2,F} - \norm{\ubar{A}_k}_{2,F}},
\end{equation}
where $Z = AG$ with $G$ drawn following a standard Gaussian distribution ($G \sim \mathcal{N}(0, \I_n)$).

In this case, since $\norm{[\I_n - \pi(Z)]\ubar{A}_k}_{2,F} \le \norm{\ubar{A}_k}_{2,F}$, we deduce that the error bounds given in Theorems~\ref{thm:3} and~\ref{thm:4} can be also  considered as upper bounds of the error quantity~\eqref{eq:old_metric} with $\Covmat(Z) = A A\tra$. In the spectral norm case, when we consider the error quantity~\eqref{eq:old_metric}, we have been able to derive an improved error bound in Theorem~\ref{thm:5}. We refer to such a bound as ``\texttt{Improved bound}" in the numerical experiments.

\begin{figure}
    \centering
    \begin{subfigure}{\textwidth}
        \centering
        \includegraphics[scale=0.5]{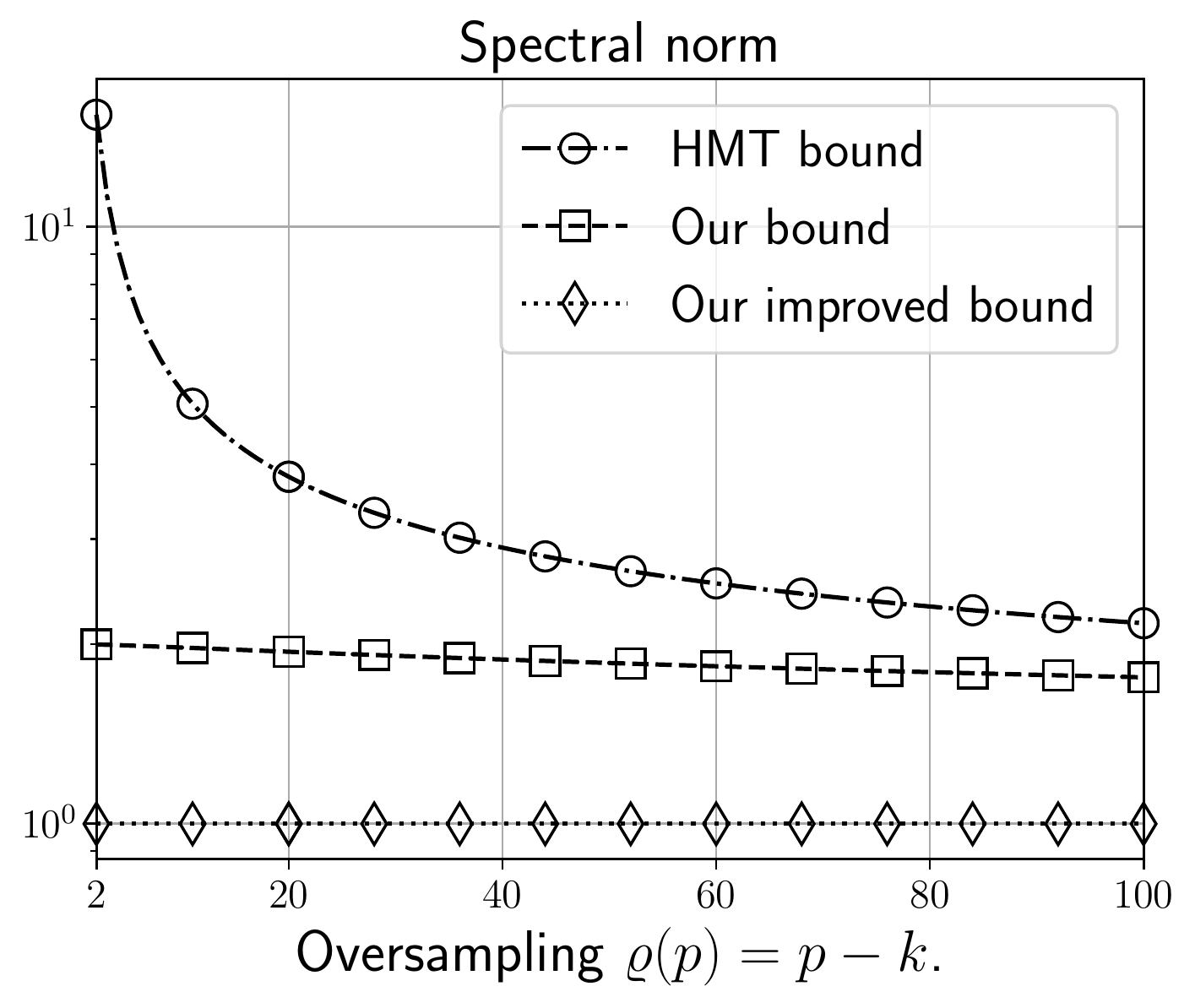}
        \includegraphics[scale=0.5]{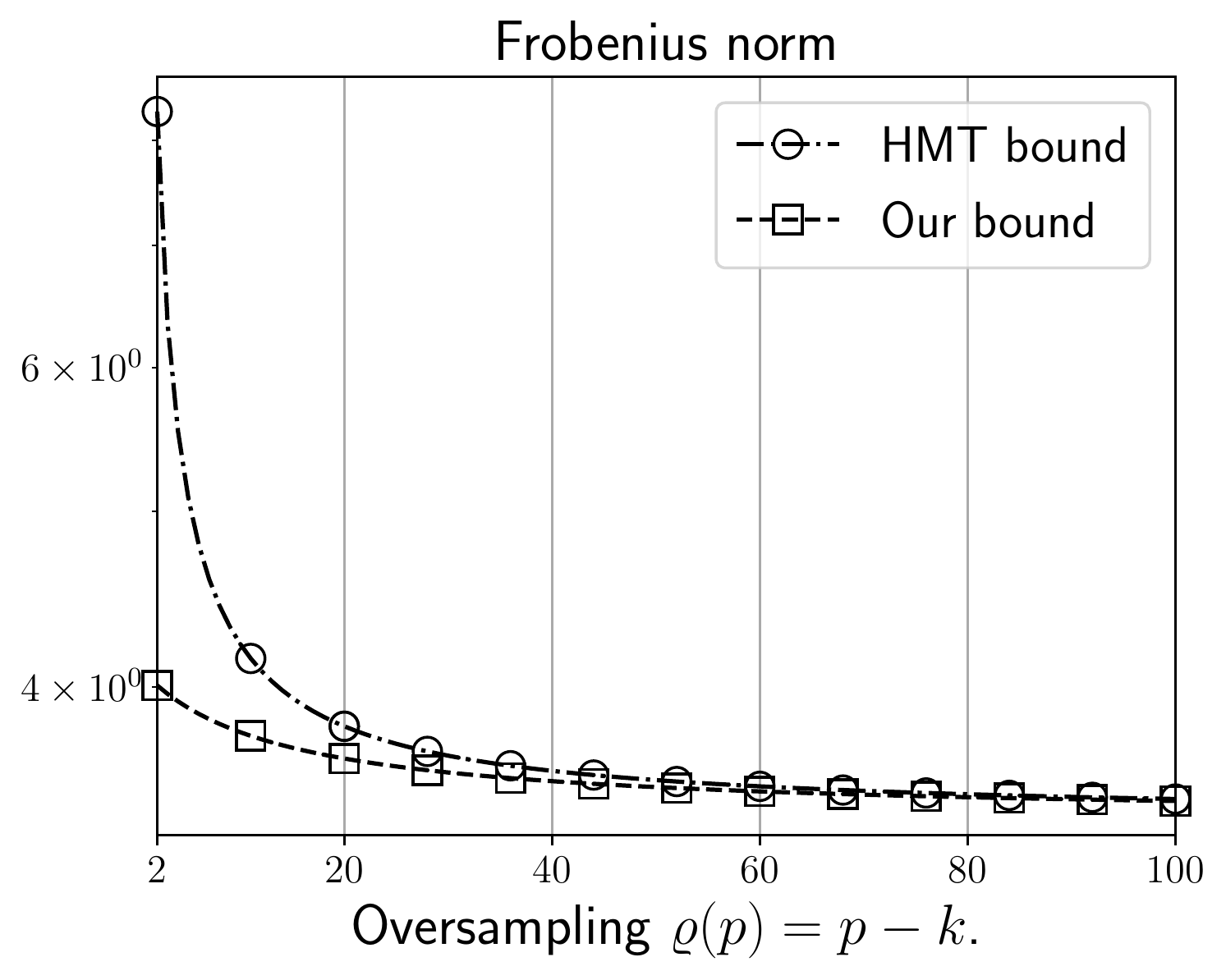}
        \caption{Results for $k=5$.}
        \label{fig:comparison_hmt:5}
    \end{subfigure}
    \begin{subfigure}{\textwidth}
    \centering
        \includegraphics[scale=0.5]{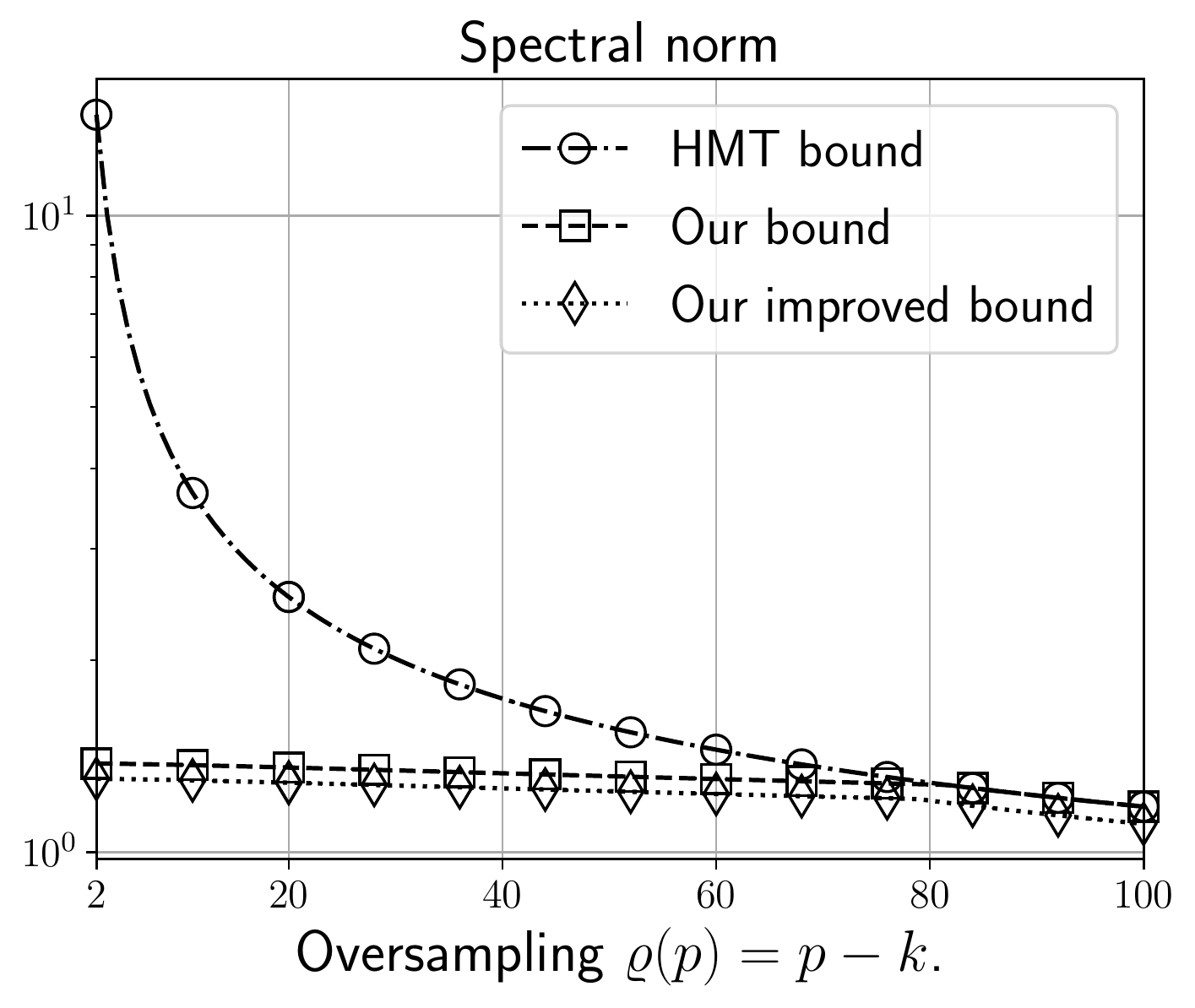}
        \includegraphics[scale=0.5]{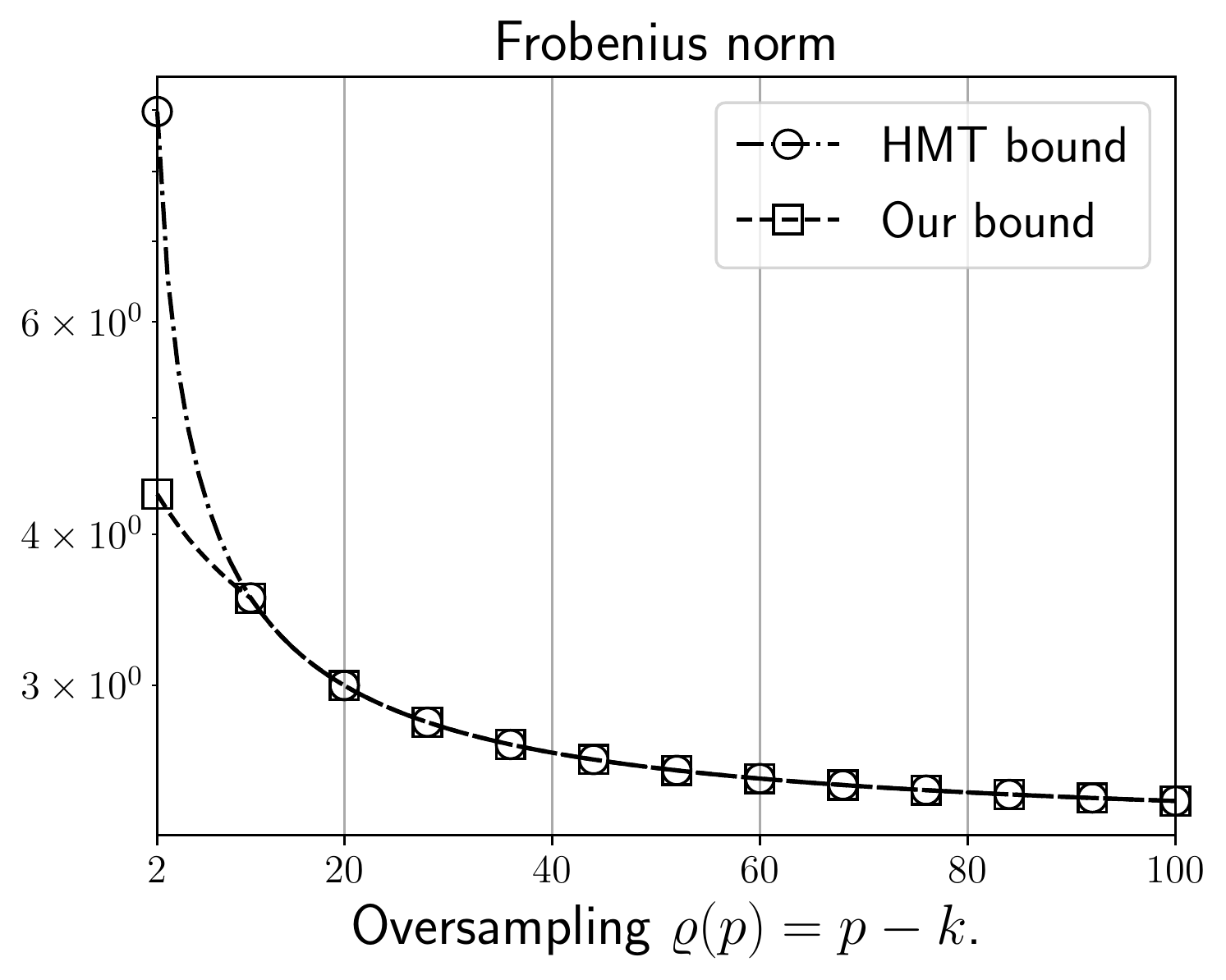}
        \caption{Results for $k=15$.}
        \label{fig:comparison_hmt:15}
    \end{subfigure}
    \caption{Comparison of different bounds for the error quantity~\eqref{eq:old_metric} using two different target ranks $k=5$ (top) and $k=15$ (bottom), spectral norm (left) and Frobenius norm (right).}
    \label{fig:comparison_hmt}
\end{figure}

Figure~\ref{fig:comparison_hmt} shows a comparison of the three error bounds for the error quantity~\eqref{eq:old_metric} using two different target ranks for $k$ in both norms. In the spectral norm case, for $k=5$, ``\texttt{Our improved bound}" outperforms by far the other error bounds (i.e., ``\texttt{HMT bound}" and ``\texttt{Our bound}"). The ``\texttt{HMT bound}" is in particular found to be  very loose for small values of the oversampling parameter $\varrho(p)$. When the target rank $k$ is getting larger (i.e., $k=15$), ``\texttt{Our bound}" and ``\texttt{Our improved bound}" behave similarly (with a slight advantage for ``\texttt{Our improved bound}"). Also, for $\varrho(p) \approx 80$, we observe a break in the curves related to both bounds. This is mainly due to a change in the minimum in our bounds. The ``\texttt{HMT bound}" is very loose for small value of $p$, but as far as the sample parameter gets larger, the bound gets close to our two error bounds.

In the Frobenius norm case, as shown in Figure~\ref{fig:comparison_hmt}, we again  notice the clear benefit of ``\texttt{Our bound}" which, compared to  ``\texttt{HMT bound}", is leading to a moderate overestimation of the error, in particular for small values of the oversampling parameter. The asymptotic behaviour of the error bounds is consistent with the fact that, for a large value of $p$, ``\texttt{Our bound}" reduces to the ``\texttt{HMT bound}" in this norm. For further comparison, we have also plotted the behaviour of the error bounds proposed by Boullé and Townsend (``\texttt{BT bound}")~\cite[Proposition 6]{BoulleTownsend_GeneralizationRandomizedSingular_2022} using $K = I_n$ in their general setting.

\subsection{Error bounds for the Randomized Singular Value Decomposition}
\label{sec:numerical:expect_comparison:RSVD}
Finally, we illustrate the relevance of our error bounds in the context of the Randomized Singular Value Decomposition.  
We consider the quantity of interest (\ref{eq:old_metric}) now with $Z = A_q G$, with $A_q = (A A\tra)^q A$ for a given $q \in \N$ and $G$ drawn following a standard Gaussian distribution ($G \sim \mathcal{N}(0, \I_n)$). We later consider the single-pass case (i.e., $q=0$) and the power scheme for two different values of the iteration (i.e., $q=1$ and $q=2$).

We denote the error bound~\cite[Corollary 10.10]{HalkoMartinssonEtAl_FindingStructureRandomness_2011} by ``\texttt{HMT bound}" in this section. To the best of our knowledge, in the context of the power iteration scheme for the Randomized Singular Value Decomposition, the ``\texttt{HMT bound}" has been derived only in the spectral norm. We note that we have been able to provide an error analysis for both the Frobenius and the spectral norms. Hence, we later denote the error bound given in Corollary~\ref{thm:3:rsvd} by ``\texttt{Our bound}".  In the spectral norm case, only ``\texttt{Our improved bound}" given in Corollary~\ref{thm:5:rsvd} is considered in the following, since it is found to outperform ``\texttt{Our bound}".

\begin{figure}
    \centering
    \begin{subfigure}{\textwidth}
        \includegraphics[scale=0.5]{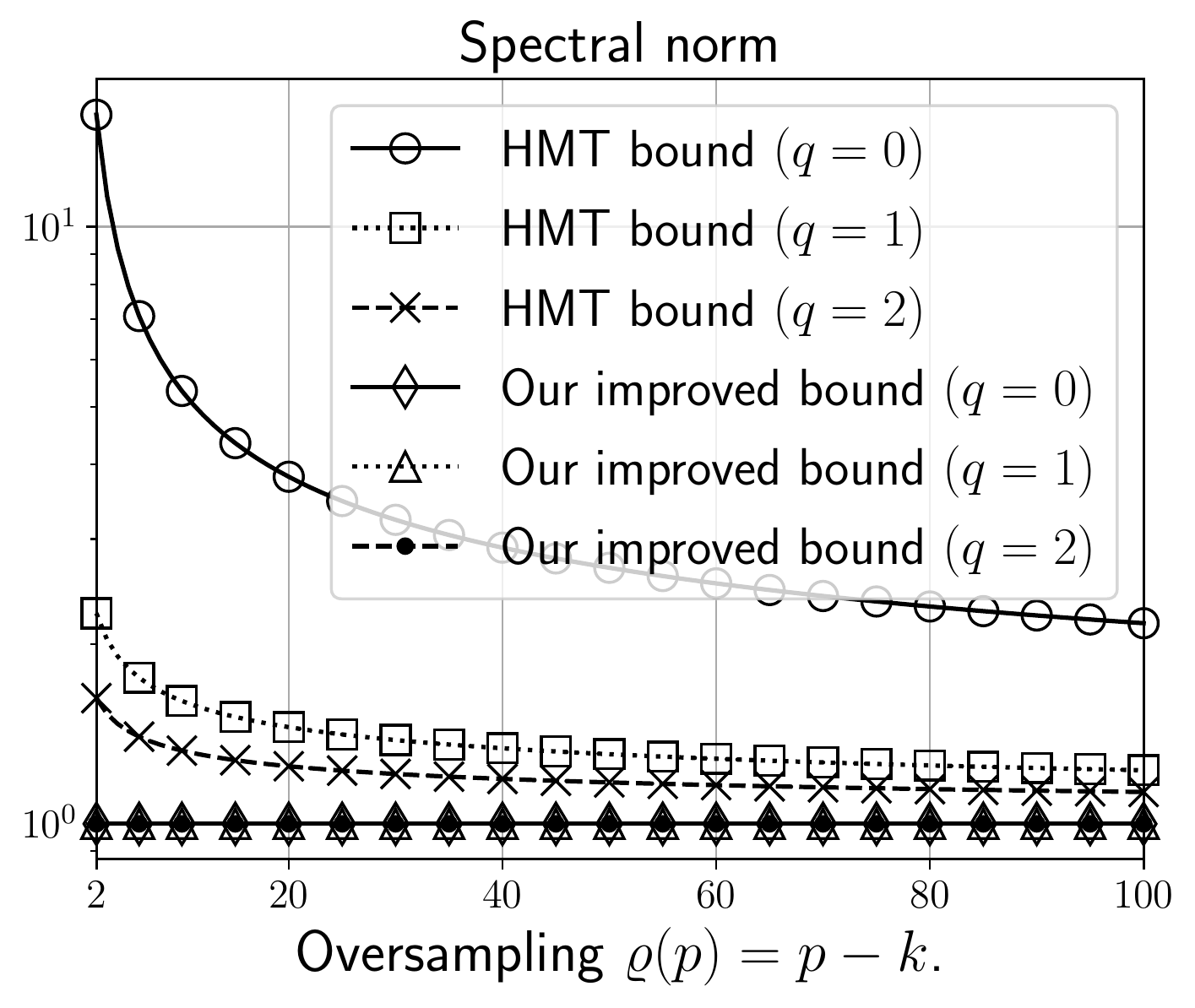}
        \includegraphics[scale=0.5]{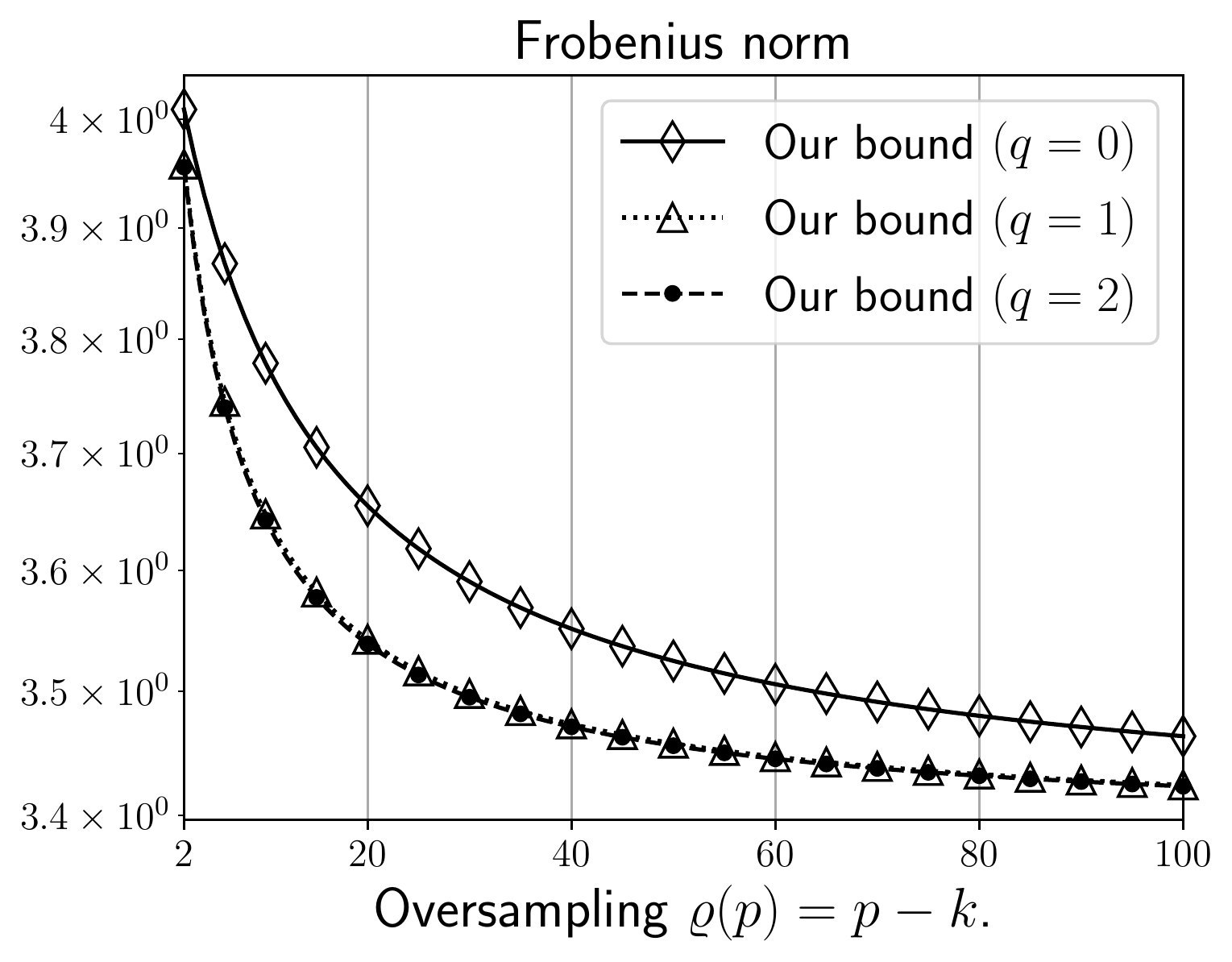}
        \caption{Results for $k=5$.}
        \label{fig:power_iteration:5}
    \end{subfigure}
    \begin{subfigure}{\textwidth}
        \includegraphics[scale=0.5]{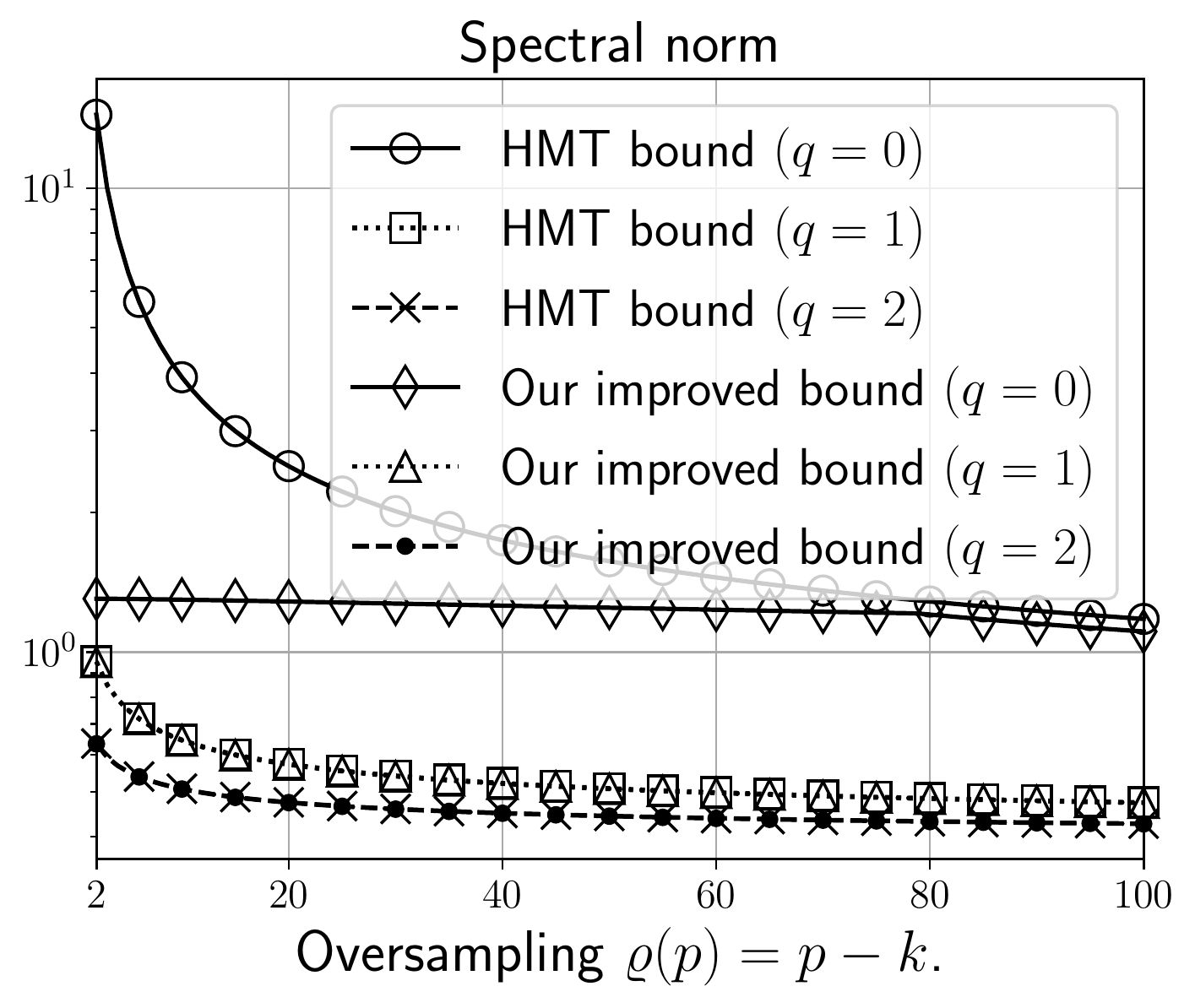}
        \includegraphics[scale=0.5]{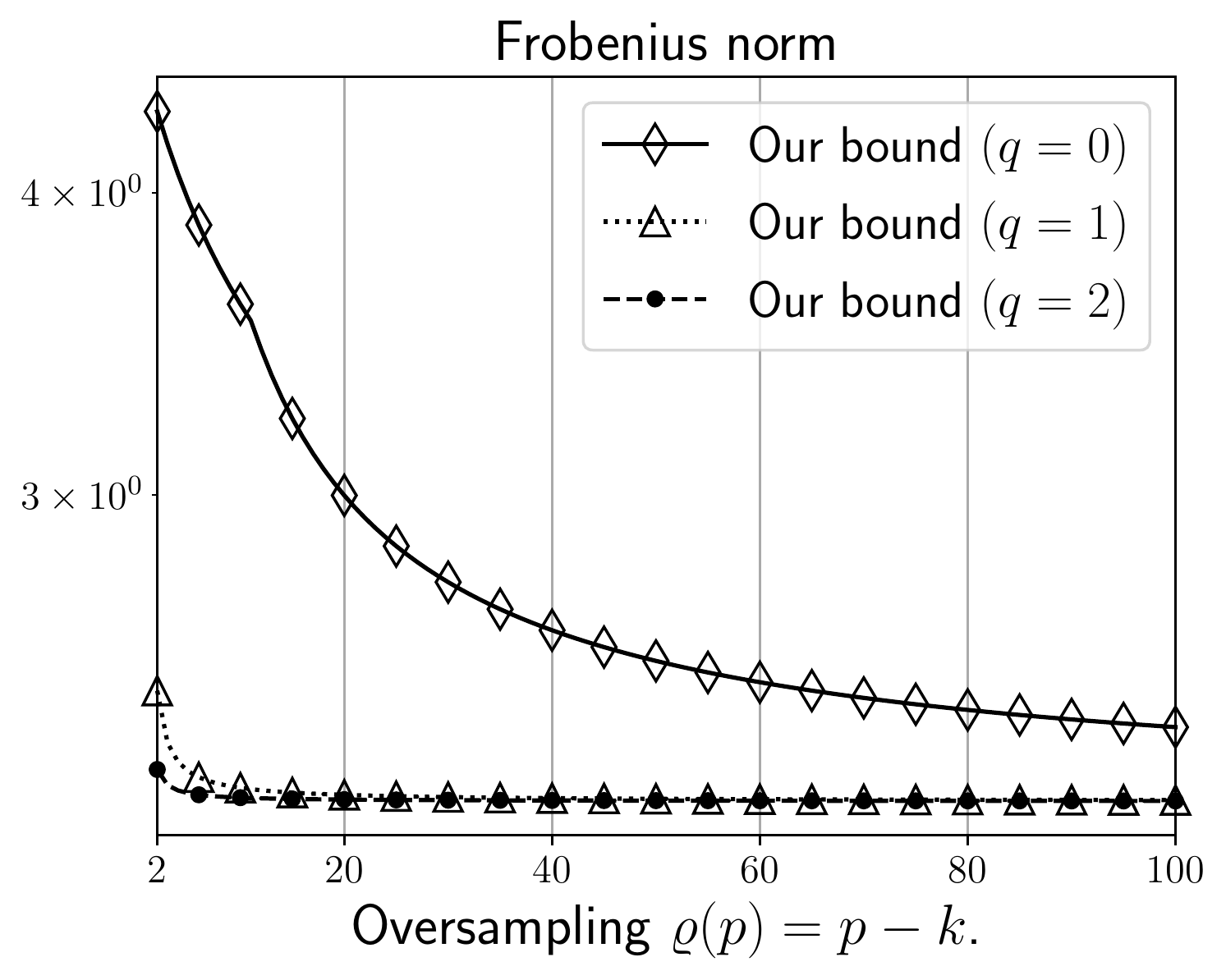}
        \caption{Results for $k=15$.}
        \label{fig:power_iteration:15}
    \end{subfigure}
    \caption{Randomized Singular Value Decomposition: comparison of different bounds for the error quantity~\eqref{eq:old_metric} for two different target ranks $k=5$ (top) and $k=15$ (bottom), spectral norm (left) and Frobenius norm (right).}
    \label{fig:power_iteration}
    \vspace{-.3cm}
\end{figure}

Figure~\ref{fig:power_iteration} shows the comparison between the different bounds related to the error quantity~\eqref{eq:old_metric}, in the context of the Randomized Singular Value Decomposition. First, as expected, we clearly identify that increasing the value of $q$ in the power iteration scheme does lead to  a strong improvement in the tightness of all the error bounds. In fact, while the error bounds are above $10^0$ when $q=0$, they become much smaller when $q \geq 1$. In the spectral norm case, the advantage of ``\texttt{Our improved bound}" over the ``\texttt{HMT bound}" is clear. We also note that for a large value of the target rank (i.e., $k=15$), the gap between the two error bounds gets smaller when $q\ge 1$. In the Frobenius norm case, we observe that the power iteration scheme is indeed very profitable, in particular for the large rank case, i.e., $k=15$. A value of $q=1$ seems to be sufficient to get optimal error bounds, since the convergence towards the optimal value $\norm{\ubar{A}_k}_F$ is almost immediate. In terms of computational cost, our numerical experiments suggest that performing only $q=1$ iteration is a very satisfactory trade-off in this test case.

\section{Conclusions} \label{sec:5}

We have analyzed theoretically the low-rank approximation to a given matrix in both the spectral and Frobenius norms. First, we have derived in Theorems~\ref{thm:1} and~\ref{thm:2} deterministic error bounds that hold with some minimal assumptions. Second, we have derived error bounds in expectation in the non-standard Gaussian case, assuming a  non-trivial mean and a general covariance matrix for the random matrix variable (Theorems~\ref{thm:3},~\ref{thm:4} and~\ref{thm:5}). This analysis generalizes and improves the error bounds proposed in~\cite{HalkoMartinssonEtAl_FindingStructureRandomness_2011}. Then, we have applied our analysis to the Randomized Singular Value Decomposition and have deduced the related error bounds in expectation (Corollaries~\ref{thm:3:rsvd},~\ref{thm:4:rsvd} and~\ref{thm:5:rsvd}). Numerical experiments on a synthetic test case have shown the tightness of the new error bounds.

In a near future, we plan to derive the probabilistic error bounds related to the error analysis proposed in Section~\ref{sec:2} and specialize them to the case of algorithms based on randomized subspace iterations. More specifically,  we aim at extending our analysis to the Generalized Singular Value Decomposition~\cite{PaigeSaunders_GeneralizedSingularValue_1981, VanLoan_GeneralizingSingularValue_1976} and to the generalized Nyström method~\cite{Nakatsukasa_FastStableRandomized_2020}, due to possible applications in weighted inverse problems, uncertainty quantification and model reduction to name a few. In all these applications, exploiting prior information on the singular vectors of $A$ to define the covariance matrix $\Covmat(Z)$ is expected to be beneficial for the accuracy of the low-rank approximation.

\bibliographystyle{siamplain}
\bibliography{bibliography}

\begin{thebibliography}{10}

\bibitem{BoulleTownsend_GeneralizationRandomizedSingular_2022}
{\sc N.~Boull{\'e} and A.~Townsend}, {\em A generalization of the randomized
  singular value decomposition}, arXiv:2105.13052 [cs, math, stat],  (2022),
  \url{https://arxiv.org/abs/2105.13052}.

\bibitem{BoulleTownsend_LearningEllipticPartial_2022}
{\sc N.~Boull{\'e} and A.~Townsend}, {\em Learning {{Elliptic Partial
  Differential Equations}} with {{Randomized Linear Algebra}}}, Foundations of
  Computational Mathematics,  (2022),
  \url{https://doi.org/10.1007/s10208-022-09556-w}.

\bibitem{BoutsidisMahoneyEtAl_ImprovedApproximationAlgorithm_2009}
{\sc C.~Boutsidis, M.~W. Mahoney, and P.~Drineas}, {\em An {{Improved
  Approximation Algorithm}} for the {{Column Subset Selection Problem}}}, in
  Proceedings of the {{Twentieth Annual ACM-SIAM Symposium}} on {{Discrete
  Algorithms}}, {Society for Industrial and Applied Mathematics}, Jan. 2009,
  pp.~968--977, \url{https://doi.org/10.1137/1.9781611973068.105}.

\bibitem{DrineasIpsen_LowRankMatrixApproximations_2019}
{\sc P.~Drineas and I.~C.~F. Ipsen}, {\em Low-{{Rank Matrix Approximations Do
  Not Need}} a {{Singular Value Gap}}}, SIAM Journal on Matrix Analysis and
  Applications, 40 (2019), pp.~299--319,
  \url{https://doi.org/10.1137/18M1163658}.

\bibitem{EckartYoung_ApproximationOneMatrix_1936}
{\sc C.~Eckart and G.~Young}, {\em The approximation of one matrix by another
  of lower rank}, Psychometrika, 1 (1936), pp.~211--218,
  \url{https://doi.org/10.1007/BF02288367}.

\bibitem{FengZhang_RankRandomMatrix_2007}
{\sc X.~Feng and Z.~Zhang}, {\em The rank of a random matrix}, Applied
  Mathematics and Computation, 185 (2007), pp.~689--694,
  \url{https://doi.org/10.1016/j.amc.2006.07.076}.

\bibitem{GolubVanLoan_MatrixComputations_1996}
{\sc G.~H. Golub and C.~F. Van~Loan}, {\em Matrix Computations}, Johns
  {{Hopkins}} Series in the Mathematical Sciences, {The Johns Hopkins
  University Press}, Baltimore, 4th~ed., 2013.

\bibitem{Gu_SubspaceIterationRandomization_2015}
{\sc M.~Gu}, {\em Subspace {{Iteration Randomization}} and {{Singular Value
  Problems}}}, SIAM Journal on Scientific Computing, 37 (2015),
  pp.~A1139--A1173, \url{https://doi.org/10.1137/130938700}.

\bibitem{Gut_IntermediateCourseProbability_2009}
{\sc A.~Gut}, {\em An {{Intermediate Course}} in {{Probability}}}, Springer
  Texts in Statistics, {Springer}, {Dordrecht New York, NY}, 2nd ed~ed., 2009,
  \url{https://doi.org/10.1007/978-1-4419-0162-0}.

\bibitem{Hackbusch_HierarchicalMatricesAlgorithms_2015}
{\sc W.~Hackbusch}, {\em Hierarchical {{Matrices}}: {{Algorithms}} and
  {{Analysis}}}, no.~49 in Hierarchical Matrices, {Springer Berlin Heidelberg},
  {Berlin, Heidelberg}, 1st~ed., 2015,
  \url{https://doi.org/10.1007/978-3-662-47324-5}.

\bibitem{HalkoMartinssonEtAl_FindingStructureRandomness_2011}
{\sc N.~Halko, P.-G. Martinsson, and J.~A. Tropp}, {\em Finding {{Structure}}
  with {{Randomness}}: {{Probabilistic Algorithms}} for {{Constructing
  Approximate Matrix Decompositions}}}, SIAM Review, 53 (2011), pp.~217--288,
  \url{https://doi.org/10.1137/090771806}.

\bibitem{Halko_RandomizedMethodsComputing_2012}
{\sc N.~P. Halko}, {\em Randomized Methods for Computing Low-Rank
  Approximations of Matrices}, PhD thesis, University of Colorado, {Boulder,
  CO, USA}, 2012.

\bibitem{Higham_FunctionsMatricesTheory_2008}
{\sc N.~J. Higham}, {\em Functions of Matrices: Theory and Computation},
  {Society for Industrial and Applied Mathematics}, {Philadelphia}, 2008.

\bibitem{HornJohnson_MatrixAnalysis_2012}
{\sc R.~A. Horn and C.~R. Johnson}, {\em Matrix Analysis}, {Cambridge
  University Press}, {Cambridge}, 2nd~ed., 2013.

\bibitem{Mahoney_RandomizedAlgorithmsMatrices_2010}
{\sc M.~W. Mahoney}, {\em Randomized {{Algorithms}} for {{Matrices}} and
  {{Data}}}, Foundations and Trends\textregistered{} in Machine Learning, 3
  (2011), pp.~123--224, \url{https://doi.org/10.1561/2200000035}.

\bibitem{MartinssonTropp_RandomizedNumericalLinear_2020}
{\sc P.-G. Martinsson and J.~A. Tropp}, {\em Randomized numerical linear
  algebra: {{Foundations}} and algorithms}, Acta Numerica, 29 (2020),
  pp.~403--572, \url{https://doi.org/10.1017/S0962492920000021}.

\bibitem{Muirhead_AspectsMultivariateStatistical_1982}
{\sc R.~J. Muirhead}, {\em Aspects of Multivariate Statistical Theory}, Wiley
  Series in Probability and Mathematical Statistics, {Wiley}, {New York}, 2009.

\bibitem{Nakatsukasa_FastStableRandomized_2020}
{\sc Y.~Nakatsukasa}, {\em Fast and stable randomized low-rank matrix
  approximation}, arXiv:2009.11392 [cs, math],  (2020),
  \url{https://arxiv.org/abs/2009.11392}.

\bibitem{PaigeSaunders_GeneralizedSingularValue_1981}
{\sc C.~C. Paige and M.~A. Saunders}, {\em Towards a {{Generalized Singular
  Value Decomposition}}}, SIAM Journal on Numerical Analysis, 18 (1981),
  pp.~398--405, \url{https://doi.org/10.1137/0718026}.

\bibitem{RokhlinSzlamEtAl_RandomizedAlgorithmPrincipal_2010}
{\sc V.~Rokhlin, A.~Szlam, and M.~Tygert}, {\em A {{Randomized Algorithm}} for
  {{Principal Component Analysis}}}, SIAM Journal on Matrix Analysis and
  Applications, 31 (2010), pp.~1100--1124,
  \url{https://doi.org/10.1137/080736417}.

\bibitem{Saibaba_RandomizedSubspaceIteration_2019}
{\sc A.~K. Saibaba}, {\em Randomized {{Subspace Iteration}}: {{Analysis}} of
  {{Canonical Angles}} and {{Unitarily Invariant Norms}}}, SIAM Journal on
  Matrix Analysis and Applications, 40 (2019), pp.~23--48,
  \url{https://doi.org/10.1137/18M1179432}.

\bibitem{SaibabaHartEtAl_RandomizedAlgorithmsGeneralized_2021}
{\sc A.~K. Saibaba, J.~Hart, and B.~van Bloemen~Waanders}, {\em Randomized
  algorithms for generalized singular value decomposition with application to
  sensitivity analysis}, Numerical Linear Algebra with Applications, 28 (2021),
  \url{https://doi.org/10.1002/nla.2364}.

\bibitem{TroppYurtseverEtAl_PracticalSketchingAlgorithms_2017}
{\sc J.~A. Tropp, A.~Yurtsever, M.~Udell, and V.~Cevher}, {\em Practical
  {{Sketching Algorithms}} for {{Low-Rank Matrix Approximation}}}, SIAM Journal
  on Matrix Analysis and Applications, 38 (2017), pp.~1454--1485,
  \url{https://doi.org/10.1137/17M1111590}.

\bibitem{VanLoan_GeneralizingSingularValue_1976}
{\sc C.~F. Van~Loan}, {\em Generalizing the {{Singular Value Decomposition}}},
  SIAM Journal on Numerical Analysis, 13 (1976), pp.~76--83,
  \url{https://doi.org/10.1137/0713009}.

\bibitem{Woodruff_SketchingToolNumerical_2014}
{\sc D.~P. Woodruff}, {\em Sketching as a {{Tool}} for {{Numerical Linear
  Algebra}}}, Foundations and Trends\textregistered{} in Theoretical Computer
  Science, 10 (2014), pp.~1--157, \url{https://doi.org/10.1561/0400000060}.

\bibitem{ZhuKnyazev_AnglesSubspacesTheir_2013}
{\sc P.~Zhu and A.~V. Knyazev}, {\em Angles between subspaces and their
  tangents}, Journal of Numerical Mathematics, 21 (2013), pp.~325--340,
  \url{https://doi.org/10.1515/jnum-2013-0013}.

\end{thebibliography}

\end{document}